\documentclass[12pt]{amsart}
\usepackage{amssymb, latexsym, amsmath, amsfonts}

\newtheorem{thm}{Theorem}[section]

\newtheorem{lem}[thm]{Lemma}
\newtheorem{prop}[thm]{Proposition}
\theoremstyle{definition}

\theoremstyle{remark}
\newtheorem{rem}[thm]{Remark}
\numberwithin{equation}{section}

\newcommand{\mbf}{\mathbf}
\newcommand{\ra}{\rightarrow}

\newcommand{\ep}{\epsilon}
\newcommand{\no}{\noindent}

\newcommand{\cal}{\mathcal}

\begin{document}
\title{On isometries of the Kobayashi and Carath\'{e}odory metrics}
\thanks{}
\author{Prachi Mahajan}
\address{Department of Mathematics,
Indian Institute of Science, Bangalore 560 012, India}
\email{mittal@math.iisc.ernet.in}

\begin{abstract}
This article considers isometries of the Kobayashi and
Carath\'{e}od-ory metrics on domains in $ \mathbf{C}^n $ and the
extent to which they behave like holomorphic mappings. First we
prove a metric version of Poincar\'{e}'s theorem about
biholomorphic inequivalence of $ \mathbf{B}^n $, the unit ball in
$ \mathbf{C}^n $ and $ \Delta^n $, the unit polydisc in $
\mathbf{C}^n $ and then provide few examples which
\textit{suggest} that $ \mathbf{B}^n $ cannot be mapped
isometrically onto a product domain. In addition, we prove several
results on continuous extension of isometries $ f : D_1
\rightarrow D_2 $ to the closures under purely local assumptions
on the boundaries. As an application, we show that there is no
isometry between a strongly pseudoconvex domain in $ \mathbf{C}^2
$ and certain classes of weakly pseudoconvex finite type domains
in $ \mathbf{C}^2 $.
\end{abstract}

\maketitle

\section{Introduction}

\no The principal aim of this article is to explore the phenomenon
of the rigidity of continuous isometries of the Kobayashi and the
Carath\'{e}odory metrics. More precisely if $D, D'$ are two
domains in $\mbf C^n$ and $f : D \ra D'$ is a continuous isometry
of the Kobayashi metrics on $D, D'$, it is not known whether $f$
must necessarily be holomorphic or conjugate holomorphic. The same
question can be asked for the Carath\'{e}odory metric or for that
matter any invariant metric as well. An affirmative answer for the
Bergman metric was given in \cite{Greene&Krantz-1982} in the case
when $D$ and $D' $ are both $C^2$-smooth strongly pseudoconvex
domains in $ \mathbf{C}^n $ and this required knowledge of the
limiting behaviour of the holomorphic sectional curvatures of the
Bergman metric near strongly pseudoconvex points. In general, the
Kobayashi metric is just upper semicontinuous and therefore a
different approach will be needed for this question. The case of
continuous isometries when $D$ is smooth strongly convex and $D'$
is the unit ball was dealt with in \cite{Seshadri&Verma-2009} and
this was improved upon in \cite{Kim&Krantz-2008} to handle the
case when $D$ is a $C^{2, \ep}$-smooth strongly pseudoconvex
domain, and a common ingredient in both proofs was the use of
Lempert discs. On the other hand, it has been remarked in
\cite{Frankel-1991} that the localization of a biholomorphic
mapping between bounded domains near a given boundary point should
follow from general principles of Gromov hyperbolicity -- an
example of this can be found in \cite{Balogh&Bonk-2000}. Motivated
by such considerations it seemed natural to determine the extent
to which isometries behave like holomorphic mappings and examples
are provided by the following results. The first three theorems in
particular deal with the most natural and ubiquitous domains,
namely the unit ball and the polydisc (and product domains in
general) and each statement leads to the next in a natural
succession. Instead of providing a general statement (i.e.,
Theorem \ref{0.3}) only, we have instead focussed on assertions
that gradually lead up to it, for this lays bare the arguments
used in a systematic manner.

\begin{thm} \label{0.1}
There is no $ C^1$-Kobayashi or Carath\'{e}odory isometry between
$ \mathbf{B}^n $, the unit ball in $ \mathbf{C}^n $ and $
\Delta^n$, the unit polydisc in $ \mathbf{C}^n $ for any $ n > 1
$.
\end{thm}

\no Several remarks are in order here. Firstly, by a $
{C}^0$-Kobayashi (Carath\'{e}odory or inner Carath\'{e}odory)
isometry we mean a distance preserving bijection between the
metric spaces $ ((D_1, d_{D_1} )$ and $ ( D_2, d_{D_2} )$ ( $(D_1,
c_{D_1}) $ and $ (D_2, c_{D_2}) $;  $(D_1, c^i_{D_1}) $ and $
(D_2, c^i_{D_2}) $ respectively). Here $ d_D, c_D $ and $ c_D^i $
denote the Kobayashi, Carath\'{e}odory and inner Carath\'{e}odory
metrics respectively on the domain $ D$. For $ k \geq 1$, a $
{C}^k$-Kobayashi isometry is a $ {C}^k$-diffeomorphism $f$ from $
D_1$ onto $ D_2$ with $ f^*( F^K_{D_2}) = F^K_{D_1}$ where $
F^K_{D_1} $ and $ F^K_{D_2} $ denote the infinitesimal Kobayashi
metrics on $ D_1 $ and $ D_2 $ respectively. Second, note that
isometries are continuous when the domains are Kobayashi
hyperbolic for in this case, the topology induced by the Kobayashi
metric coincides with the intrinsic topology of the domain. Third,
Theorem \ref{0.1} may be regarded as a version of Poincar\'{e}'s
theorem about biholomorphic inequivalence of $ \mathbf{B}^n $ and
$ \Delta^n $ for isometries. As can be expected, the main step in
proving the above result is to show that the $ C^1$-smooth
isometry, if it exists, is indeed a biholomorphic mapping to
arrive at a contradiction. The proof of this is based on
differential geometric considerations in particular the theorem of
Myers--Steenrod --- as in \cite{Seshadri&Verma-2009} and the fact
that the Kobayashi metric of the ball is a smooth K\"{a}hler
metric of constant negative sectional curvature $ -4 $ plays a key
role.

\begin{thm} \label{0.2}
There is no $ C^1$-Kobayashi or Carath\'{e}odory isometry between
$ \mathbf{B}^n $ and the product of $ m $ Euclidean balls $
\mathbf{B}^{n_1} \times \mathbf{B}^{n_2} \times \ldots \times
\mathbf{B}^{n_m} $ for any $ 2 \leq m \leq n  $ where $ n = n_1 +
n_2 + \ldots + n_m $.
\end{thm}

\noindent A similar situation as in the above theorem was
considered in Proposition 2.2.8 of \cite{Jarnicki&Pflug} (see also
\cite{Kuczumow&Ray-1988}); the emphasis here being a different
approach which is valid in a more general context and one example
is provided by the following:

\begin{thm} \label{0.3}
There is no $ C^1$-Kobayashi or Carath\'{e}odory isometry between
$ \mathbf{B}^n $ and the product of $ m $ domains  $ D_1 \times
D_2 \times \ldots \times D_m $ for any $ 2 \leq m \leq n  $ where
each $ D_i $ is a bounded strongly convex domain in $
\mathbf{C}^{n_i}$ with $C^6 $-smooth boundary and $ n = n_1 + n_2
+ \ldots + n_m $.
\end{thm}

\noindent While the proof of Theorem \ref{0.2} is along the same
lines as that of Theorem \ref{0.1}, the proof of Theorem \ref{0.3}
requires the existence of complex geodesics and certain degree of
smoothness of the Kobayashi metric and hence we restrict to $
C^6$-smooth strongly convex domains. The proofs of Theorem
\ref{0.1}, \ref{0.2} and \ref{0.3} are contained in section 3. It
must be mentioned that the above results are motivated by the well
known fact that there does not exist a proper holomorphic mapping
from a product domain onto $ \mathbf{B}^n $ for any $n > 1 $.

\medskip

\noindent This article also considers the question of continuous
extendability up to the boundary of continuous isometries between
domains in $ \mathbf{C}^n $. Here is a prototype statement that
can be proved.

\begin{thm} \label{1}
Let $ f : D_1 \rightarrow D_2 $ be a continuous Kobayashi isometry
between two bounded domains in $ \mathbf{C}^2$. Let $ p^0 $ and $
q^0 $ be points on $ \partial D_1 $ and $ \partial D_2 $
respectively. Assume that $ \partial D_1 $ is $
{C}^{\infty}$-smooth weakly pseudoconvex of finite type near $ p^0
$ and that $ \partial D_2 $ is $ {C}^2$-smooth strongly
pseudoconvex in a neighbourhood $ U_2 $ of $ q^0 $. Suppose that $
q^0 $ belongs to the cluster set of $p^0 $ under $f$. Then $f$
extends as a continuous mapping to a neighbourhood of $ p^0 $ in $
\overline{D}_1 $.
\end{thm}

\noindent It should be noted that there are only purely local
assumptions on $ D_1 $ and $ D_2 $ -- in particular, the domains
are not assumed to be pseudoconvex away from $ p^0 $ and $ q^0 $
and there are no global smoothness assumptions on the boundaries.
The above result is proved using the global estimates on the
Kobayashi metric near weakly pseudoconvex boundary points of
finite type from \cite{Herbort-2005}. This is done in Proposition
\ref{S} and Proposition \ref{Q}. It is worthwhile mentioning that
other relevant theorems of this nature for proper holomorphic
mappings between strongly pseudoconvex domains were proved by
Forstneric and Rosay (\cite{Forstneric&Rosay-1987}) using global
estimates on the Kobayashi metric. As an application of Theorem
\ref{1} we get:

\begin{thm} \label{2}
Let $ D_1$ and $ D_2$ be bounded domains in $ \mathbf{C}^2$. Let $
p^0 = (0,0) $ and $ q^0 $ be points on $ \partial D_1 $ and $
\partial D_2 $ respectively. Assume that $ \partial D_1 $ in a neighbourhood $ U_1
$ of the origin is defined by
\[
\big \{ (z_1, z_2) \in \mathbf{C}^2 : 2 \Re z_2 + | z_1 |^{2m} +
o( |z_1|^{2m} + \Im z_2 ) < 0 \big \}
\]
where $ m > 1 $ is  a positive integer and that $ \partial D_2 $
is $ {C}^2$-smooth strongly pseudoconvex in a neighbourhood $ U_2
$ of a point $ q^0 \in \partial D_2 $. Then there cannot be a
continuous Kobayashi isometry $f$ from $ D_1$ onto $D_2$ with the
property that $ q^0 $ belongs to the cluster set of $ p^0$ under
$f$.
\end{thm}

\noindent Theorem \ref{2} dispenses with the assumption of having
a global biholomorphic mapping and replaces it with a global
Kobayashi isometry at the expense of restricting to certain
classes of weakly pseudoconvex finite type domains in $
\mathbf{C}^2 $. A particularly useful strategy to investigate this
type of results in the holomorphic category has been Pinchuk's
scaling technique (cf. \cite{Pinchuk-1991}). Scaling $ D_1$ near $
p^0 $ with respect to a sequence of points that converges to $ p^0
$ along the inner normal yields a limit domain of the form
\[
D_{1, \infty} = \Big \{ (z_1, z_2) \in \mathbf{C}^2 : 2 \Re z_2 +
|z_1|^{2m} < 0 \Big \}
\]
for which the Kobayashi metric has some smoothness
(\cite{Ma-1995}). It is for this reason that we restrict attention
to domains with a defining function as described in Theorem
\ref{2}. In trying to adapt the scaling methods in our situation,
the `normality' of the scaled isometries needs to be established.
This requires the stability of the integrated Kobayashi distance
under scaling of a given strongly pseudoconvex domain (this was
done in \cite{Seshadri&Verma-2006}) and a weakly pseudoconvex
finite type domain in $ \mathbf{C}^2 $ -- this was developed in
\cite{Mahajan&Verma} for a different application and we intend to
use it here as well. The conclusion then would be that the limit
of scaled isometries exists and yields a continuous isometry
between the corresponding model domains, i.e., the ellipsoid $
D_{1, \infty} $ and the ball $ \mathbf{B}^2 $. Another difficulty
is that unlike the holomorphic case the restrictions of Kobayashi
isometries to subdomains are not isometries with respect to the
Kobayashi metric of the subdomain. The end game lies in showing
that the continuous isometry is holomorphic and this is done using
the techniques employed in the proof of Theorem \ref{0.1} --
\ref{0.3}.

\medskip

\noindent Several other statements about the continuous
extendability of continuous isometries are possible -- these
relate to isometries between either a pair of strongly
pseudoconvex domains in $ \mathbf{C}^n $ or between a pair of
weakly pseudoconvex domains of finite type in $ \mathbf{C}^2 $.
These have been stated (cf. Theorem \ref{W} and Theorem \ref{L})
and elaborated upon towards the end of section 4. These are valid
for isometries of the inner Carath\'{e}odory distance as well (cf.
Theorem \ref{U}).

\medskip

\noindent The author wishes to thank Prof. Kaushal Verma for his
encouragement and for many useful discussions throughout the
course of this work. Special thanks are also due to Prof. W.
Zwonek who provided valuable feedback after this article was put
on the math arxiv. This theme of exploring isometries was also
considered by him in \cite{Zwonek1-1993}, \cite{Zwonek2-1993} and
\cite{Zwonek-1995} and his unpublished thesis (\cite{Zwonek4}) and
our theorems \ref{0.1}, \ref{0.2} and \ref{0.3} and other more
general statements within this paradigm were considered by him in
these papers using very different methods.


\section{Notation and Terminology}

\noindent Let $\Delta$ denote the open unit disc in the complex
plane and let $d_{hyp}(a, b)$ denote the distance between two
points $a, b \in \Delta $ with respect to the hyperbolic metric.
For $ r > 0$, $ \Delta(0, r) \subset \mathbf{C} $ will be the disc
of radius $r$ around the origin and $B(z, \delta) \subset
\mathbf{C}^n$ will be the Euclidean ball of radius $\delta > 0$
around $ z $. Let $ X $ be a complex manifold of dimension $n$.
The Kobayashi and the Carath\'{e}odory distances on $X$, denoted
by $d_X$ and $c_X$ respectively, are defined as follows:

\medskip

\no Let $z \in X$ and fix $\xi$ a holomorphic tangent vector at
$z$. Define the associated infinitesimal Carath\'{e}odory and
Kobayashi metrics as
\begin{eqnarray*} F^C_X(z,\xi) & = & \sup \{ | df(z) \xi |
: f \in \mathcal{O}(X, \Delta) \} \\
\mbox{and} \qquad F^K_X(z,\xi) & = & \inf \left \lbrace
\frac{1}{\alpha} : \alpha > 0, f \in \mathcal{O}(\Delta, X) \
\mbox{with} \ f(0) = z, f'(0) = \alpha \xi \right \rbrace
\end{eqnarray*}
respectively. This induces a concept of length of a path. If
$\gamma : [0, 1] \rightarrow X$ is a piecewise smooth path, then
the Carath\'{e}odory length is given by
\[
L_X^C(\gamma) = \int _0^1 F^C_X( \gamma(t),\dot{\gamma}(t) ) \;
dt,
\]
and this in turn induces the associated inner Carath\'{e}odory
distance, namely
\[
c^i_X(p,q) = \inf L_X^C(\gamma)
\]
where the infimum is taken over all piecewise smooth curves
$\gamma$ in $X$ joining $p$ to $q$. Likewise, the Kobayashi length
of a piecewise $\mathcal{C}^1$-curve $\gamma : [0, 1] \rightarrow
X$ is given by
\[
L_X^K(\gamma) = \int _0^1 F^K_X( \gamma(t),\dot{\gamma}(t) ) \;
dt,
\]
and finally the Kobayashi distance between $p, q \in X$ is defined
as
\[
d_X(p,q) = \inf L_X(\gamma)
\]
where the infimum is taken over all piecewise differentiable
curves $\gamma$ in $X$ joining $p$ to $q$. Recall that $X$ is taut
if $\cal O(\Delta, X)$ is a normal family.

\medskip

\no The Carath\'{e}odory distance $c_X$ between $p, q \in X$ is
defined by setting
\[
c_X (p,q) = \sup_f d_{hyp}\big( f(p),f(q) \big)
\]
where the supremum is taken over the family of all holomorphic
mappings $f : X \ra \Delta$.

\medskip

\noindent A domain $ D \subset \mathbf{C}^n $ with $ C^2$-smooth
boundary is said to be \textit{strongly convex} if there is a
defining function $ \rho $ for $ \partial D $ such that the real
Hessian of $ \rho $ is positive definite as a bilinear form on $
T_p (\partial D) $ for every $ p \in \partial D $.

\medskip

\noindent Let $ D \subset \mathbf{C}^n $ be a bounded domain. A
holomorphic mapping $ \phi : \Delta \rightarrow D $ is said to be
an \textit{extremal disc} or a \textit{complex geodesic} for the
Kobayashi distance if it is distance preserving, i.e., $ d_D \big(
\phi(p), \phi(q) \big) = d_{\Delta}(p, q) $ for all $ p, q $ in $
\Delta $.

\section{Isometries versus biholomorphisms}

\noindent \textit{Proof of Theorem \ref{0.1}:} Suppose there
exists a $C^1 $-isometry $f : \big( \Delta^n, d_{\Delta^n} \big)
\rightarrow \big( \mathbf{B}^n, d_{\mathbf{B}^n} \big) $. Consider
the restriction of $f$ to lines in $ \Delta^n $ parallel the
coordinate axes. Such lines are parameterized in the following way
-- For $ ( a_1, \ldots, a_{n-1} ) \in \Delta^{n-1} $ fixed,
consider $ \phi^1_{a_1 \ldots a_{n-1}} : \Delta \rightarrow
\Delta^n $ defined by
\[
\phi^1_{a_1 \ldots a_{n-1}} (z) = (z, a_1, \ldots, a_{n-1})
\]
for $ z$ in $ \Delta $. Then using the explicit form of the
Kobayashi distance on the polydisc, it follows that $ \phi^1_{a_1
\ldots a_{n-1}} $ is distance preserving, i.e.,
\[
d_{\Delta} (z,w) = d_{\Delta^n} \big ( \phi^1_{a_1 \ldots
a_{n-1}}(z), \phi^1_{a_1 \ldots a_{n-1}}(w) \big)
\]
for all $ z, w $ in $ \Delta $. Indeed,
\begin{eqnarray*}
d_{\Delta^n} \big ( \phi^1_{a_1 \ldots a_{n-1}}(z), \phi^1_{a_1
\ldots a_{n-1}}(w) \big) & = & d_{\Delta^n} \big( (z, a_1, \ldots,
a_{n-1}), (w, a_1, \ldots, a_{n-1}) \big) \\
& = & \max \big \lbrace d_{\Delta} (z,w), d_{\Delta}(a_1, a_1),
\ldots, d_{\Delta} ( a_{n-1}, a_{n-1} ) \big \rbrace \\
& = & d_{\Delta} (z,w).
\end{eqnarray*}
Consider the composition $ \tilde{f}^1_{a_1 \ldots a_{n-1}} = f
\circ \phi^1_{a_1 \ldots a_{n-1}} : \Delta \rightarrow
\mathbf{B}^n$. Write $ \tilde{f}^1_{a_1 \ldots a_{n-1}} =
\tilde{f}^1 $ for brevity and note that $ \tilde{f}^1 $ preserves
the Kobayashi distance. More concretely,
\[
d_{\mathbf{B}^n} \big( \tilde{f}^1(z), \tilde{f}^1(w) \big) =
d_{\Delta} (z,w)
\]
for all $z, w \in \Delta $. The proof now divides into two parts.
In the first part, we show that $ \tilde{f}^1 $ is $
C^{\infty}$-smooth. This is done by adapting the proof of the
theorem of Myers-Steenrod. In the second one we prove that $
\tilde{f}^1 $ has to be (anti)-holomorphic using ideas from
\cite{Seshadri&Verma-2009}.

\medskip

\noindent \textbf{Step I:} The mapping $ \tilde{f}^1 $ is $
C^{\infty}$-smooth.

\medskip

\noindent To establish this, let $p$ be an arbitrary point of $
\Delta $ and put $ q = \tilde{f}^1(p) $. Let $ B_r(p)$ and $B_r(q)
$ be spherical normal neighbourhoods of $ p \in \Delta $ and $ q
\in \mathbf{B}^n $ respectively. These follow from the fact that $
F^K_{\Delta} $ and $ F^K_{\mathbf{B}^n } $ are both Riemannian. We
may assume that $ \tilde{f}^1 \big( B_r(p) \big) \subset B_r(q) $.
Consider germs of integral curves through $p$ in all directions.
The goal now is to show that their image under $ \tilde{f}^1 $ are
integral curves through $q$ in all directions. Indeed, for each $
v \in T_p{\Delta} $, consider the geodesic
\[
t \rightarrow \mbox{Exp}_p \; tv, \   -r/ \big(F^K_{\Delta} (p, v)
\big)  < t < r/ \big(F^K_{\Delta} (p, v) \big).
\]
The image $ \gamma(t) = \tilde{f}^1( \mbox{Exp}_p \; tv) $ lies in
$ B_r(q) $ and has the property that
\[
d_{\mathbf{B}^{n}} \big( \gamma(t), \gamma(t') \big) = |t -t'|
F^K_{\Delta} (p, v)
\]
for all $ t, t' $ in the interval of definition. This uses the
fact that $ \tilde{f}^1 $ is distance preserving. To see that $
\gamma $ is a geodesic, we consider the point $ q = \gamma(0) $
and an arbitrary point $ Q$ on trace $ \gamma $. They can be
joined by a unique geodesic $ \sigma $ of length $
d_{\mathbf{B}^n} (q, Q) $. Let $ B_R (Q) $ be a spherical normal
neighbourhood of $ Q$ and let $s$ be any point on trace $ \gamma $
between $q$ and $Q$ such that $ s \in B_R(Q) $. Then $
d_{\mathbf{B}^n} (q,s) + d_{\mathbf{B}^n}(s, Q) =
d_{\mathbf{B}^n}(q, Q) $ by construction. If we join $q$ and $s$
by the shortest distance realizing curve, and then join $s$ and
$Q$ by the shortest distance minimizing curve, we get a piecewise
differentiable curve of length $ d_{\mathbf{B}^n}(q, Q) $. This
curve must coincide with $ \sigma $. Since $ Q$ was arbitrary on
Image($ \gamma $), this proves that $ \gamma $ is a geodesic. In
particular, $\gamma $ is differentiable.

\medskip

\noindent Let $v'$ denote the tangent vector to $ \gamma $ at the
point $q$. Consider the mapping $ g : T_p {\Delta} \rightarrow T_q
{\mathbf{B}^n} $ defined by setting $ g(v) = v' $. Observe that $
F^K_{\Delta} (p, v) = F^K_{\mathbf{B}^n} \big(q, g(v) \big) $ and
$ g(\alpha v) = \alpha g(v) $ for $ \alpha \in \mathbf{R} $ and $
v \in T_p{\Delta} $. Now, let $ v^1, v^2 \in T_p{\Delta} $ and
choose $ \rho $ such that $ F^K_{\Delta}(p, \rho v^1 ) $ and $
F^K_{\Delta}(p, \rho v^2 ) $ are both less than $r$. Let $ v^1_t =
\mbox{Exp}_p \; tv^1 $ and $ v^2_t = \mbox{Exp}_p \; tv^2 $ for $
0 \leq t \leq \rho $. Let $ F_{\Delta}^K $ and $
F_{\mathbf{B}^n}^K $ be the quadratic forms associated to
Riemannian metrics $ \langle \cdot, \cdot \rangle _p $ at $ p \in
\Delta $ and $ \langle \cdot, \cdot \rangle _q $ at $ q \in
\mathbf{B}^n $ respectively. The choice of the metric will be
clear from the context thereby avoiding ambiguities due to the
same notation used. Then from \cite{Helgason}, we see that
\[ \displaystyle\lim_{ (v_t^1, v_t^2) \rightarrow (p,p) }
\frac {F^K_{\Delta} (p, tv^1 - tv^2)}{d_{\Delta} (v^1_t, v^2_t) }
= 1
\]
so that
\[
\frac{2 \langle v^1, v^2 \rangle _p}{F^K_{\Delta} (p,
v^1)F^K_{\Delta} (p, v^2)} = \frac{ \big (F^K_{\Delta} (p, v^1)
\big) ^2 + \big( F^K_{\Delta} (p, v^2) \big)^2 }{F^K_{\Delta} (p,
v^1)F^K_{\Delta} (p, v^2)} - \frac{\big( F^K_{\Delta} (p, tv^1 - t
v^2 ) \big) ^2 }{F^K_{\Delta} (p, tv^1)F^K_{\Delta} (p, tv^2)}
\]
\[
=  \frac{ \big (F^K_{\Delta} (p, v^1) \big) ^2 + \big(
F^K_{\Delta} (p, v^2) \big)^2 }{F^K_{\Delta} (p, v^1)F^K_{\Delta}
(p, v^2)} - \displaystyle\lim_{ t \rightarrow 0} \frac
{{d_{\Delta} (v^1_t, v^2_t) }^2 }{F^K_{\Delta} (p,
tv^1)F^K_{\Delta} (p, tv^2)}.
\]
Since the right hand side is preserved by the mapping $
\tilde{f}^1 $, it follows that
\[
\langle v^1, v^2 \rangle _p = \langle g(v^1), g(v^2) \rangle_q
\]
But $ v^1 + v^2 $ is determined by the quantities $ F^K_{\Delta}
(p, v^1), F^K_{\Delta} (p, v^2) $ and $ \langle v^1, v^2 \rangle
_p $, all of which are preserved by $ g $. It follows that $ g(v^1
+ v^2) = g(v^1) + g(v^2) $ which together with the previous
properties of $ g $ shows that it is a diffeomorphism of $
T_p{\Delta} $ onto $ T_q {\mathbf{B}^n} $. On $ B_r(p) $ we have
\[
\tilde{f}^1 = \mbox{Exp}_q \circ g \circ \mbox{Exp}^{-1}_p.
\]
This exactly means that the mapping $ \tilde{f}^1 $ is linear in
exponential coordinates. Since the exponential map is smooth, $
\tilde{f}^1 $ is smooth.

\medskip

\noindent \textbf{Step II:} $ \tilde{f}^1 $ is
holomorphic/anti-holomorphic.

\medskip

\noindent Let $ J_0 $ and $ J$ denote the almost complex
structures on $ T \mathbf{B}^n $ and $ T \Delta $ respectively. It
suffices to prove that $ d \tilde{f}^1 \circ J = \pm J_0 \circ d
\tilde{f}^1 $. To do this, fix $p \in \Delta $ and let $ S_0 $ and
$ S$ denote the set of complex lines i.e. $2$-planes invariant
under $ J_0 $ and $J$ respectively. We claim that $ J $ invariant
$ 2$-planes go to $ J_0 $ invariant $2$-planes under $ d
\tilde{f}^1 $. First note that since $ \tilde{f}^1 $ is smooth,
the sectional curvature of $ \tilde{f}^1 (\Delta) $ with respect
to the metric induced by $ \tilde{f}^1 $ is equal to that of $
\Delta $ with respect to the hyperbolic metric, i.e., $-4 $. On
the other hand, since $ \tilde{f}^1 $ is distance preserving, it
takes geodesics in $ \Delta $ to geodesics in $ \mathbf{B}^n $ and
$ \tilde{f}^1 ( \Delta) $ is a totally geodesic submanifold of $
\mathbf{B}^n $. Hence, the sectional curvature of $ \tilde{f}^1(
\Delta ) $ at $ \tilde{f}^1(p), p \in \Delta $ with respect to the
induced metric is equal to the sectional curvature in the $
F^K_{\mathbf{B}^n} $-metric. This can be realized only by
holomorphic sections in the ball -- indeed, since $
F^K_{\mathbf{B}^n }$ has constant holomorphic sectional curvature
$ -4 $, at any point the sectional curvature of a $ 2$-plane $P$
spanned by an orthonormal pair of tangent vectors $ X, Y$ is $ -(
1 + 3 \langle X, J_0 Y \rangle ) $. In particular, a
two-dimensional subspace $Q$ of the tangent space at the point $
\tilde{f}^1(p) $ is in $ S_0 $ if and only if the sectional
curvature of $ Q $ is $ -4 $. This shows that complex lines are
taken to complex lines by $ d \tilde{f}^1 $.

\medskip

\noindent Consequently,  $ d \tilde{f}^1 \circ J = \pm J_0 \circ d
\tilde{f}^1 $ on any $ P \in S $. Now, using the fact that $ S$ is
connected as a subset of the Grassmann manifold of $2$-planes in $
T _p(\Delta) $ we can conclude that  $ d \tilde{f}^1 \circ J = J_0
\circ d \tilde{f}^1 $ on every $ P \in S $ or  $ d \tilde{f}^1
\circ J = - J_0 \circ d \tilde{f}^1 $ on every $ P \in S $, i.e.,
$ d \tilde{f}^1 \circ J = \pm J_0 \circ d \tilde{f}^1 $ on $ T_p (
\Delta) $. From the connectedness of $ \Delta $, it follows that $
d \tilde{f}^1 \circ J = \pm J_0 \circ d \tilde{f}^1 $ on $ T
\Delta $. This completes Step II.

\medskip

\noindent Further, recall that $f \in C^1 $ by assumption and
consequently the mapping $ (a_1, \ldots, a_{n-1}) \rightarrow
\tilde{f}^1_{a_1 \ldots a_{n-1}} $ is also $ C^1 $ -- this is the
only point in the proof that uses the $C^1$-smoothness of $f$.
Now, from the connectedness of $ \Delta^{n-1} $, we see that
either
\[
d \tilde{f}^1_{a_1 \ldots a_{n-1}} \circ J = J_0 \circ d
\tilde{f}^1_{a_1 \ldots a_{n-1}}
\]
or
\[
d \tilde{f}^1_{a_1 \ldots a_{n-1}} \circ J = - J_0 \circ d
\tilde{f}^1_{a_1 \ldots a_{n-1}}
\]
for every $ ( a_1, \ldots, a_{n-1} ) \in \Delta^{n-1} $. In other
words, either $ \tilde{f}^1_{a_1 \ldots a_{n-1}} $ is holomorphic
for every choice of $ ( a_1, \ldots, a_{n-1} ) \in \Delta^{n-1} $
or anti-holomorphic for every $ ( a_1, \ldots, a_{n-1} ) \in
\Delta^{n-1} $. Replacing $ \tilde{f}^1_{a_1 \ldots a_{n-1}} $ by
it's complex conjugate, if necessary, we may assume that $
\tilde{f}^1_{a_1 \ldots a_{n-1}} $ is holomorphic for every $ (
a_1, \ldots, a_{n-1} ) \in \Delta^{n-1} $.

\medskip

\noindent For each $ j=2, \ldots, n $, consider the composition $
\tilde{f}^j_{a_1 \ldots a_{n-1}} = f \circ \phi^j_{a_1 \ldots
a_{n-1}} $ where
\[
\phi^j_{a_1 \ldots a_{n-1}} (z) = ( a_1, \ldots, a_{j-1}, z, a_j,
\ldots, a_{n-1} )
\]
for $ z \in \Delta $ and $ ( a_1, \ldots, a_{n-1} ) \in
\Delta^{n-1} $. Now, an argument similar to the one employed in
Step I and Step II shows that $ \tilde{f}^j_{a_1 \ldots a_{n-1}} $
are holomorphic for every $ ( a_1, \ldots, a_{n-1} ) \in
\Delta^{n-1} $ and for all $ 2 \leq j \leq n $. Said differently,
this is just the assertion that $f$ is holomorphic in each
variable separately. Applying Hartog's theorem on separate
analyticity to the mapping $f$, we conclude that $f$ is
holomorphic on $ \Delta^n $. This violates the fact that there
cannot be a biholomorphism from $ \Delta^n $ onto $ \mathbf{B}^n
$. This contradiction proves the theorem for the Kobayashi metric.
Since the Kobayashi and the Carath\'{e}odory metrics coincide on $
\Delta^n $ and $ \mathbf{B}^n $, we obtain that there is no $
C^1$-Carath\'{e}odory isometry between $ \Delta^n $ and $
\mathbf{B}^n $ for any $ n > 1 $. \qed

\medskip

\noindent \textit{Proof of Theorem \ref{0.2}:} To prove this,
suppose that for some $ 2 \leq m \leq n $, there exists a $C^1
$-Kobayashi isometry $f : \mathbf{B}^{n_1} \times \mathbf{B}^{n_2}
\times \ldots \times \mathbf{B}^{n_m} \rightarrow \mathbf{B}^n $.
Now, let $ L_{1,j} $ denote the $j^{{th}} $-coordinate axis in $
\mathbf{C}^{n_1} $ and let $ \phi_{1,j} $ be a holomorphic
parametrization of the intersection of the complex line $ L_{1,j}
$ with $ \mathbf{B}^{n_1} $ and it can be checked that it is an
isometric immersion from $ \Delta $ into $ \mathbf{B}^{n_1} $,
i.e., for every $ 1 \leq j \leq n_1 $,
\begin{equation}\label{D0}
d_{\mathbf{B}^{n_1} } \big( \phi_{1,j} (p), \phi_{1,j}(q) \big) =
d_{\Delta} (p,q)
\end{equation}
for all $p, q \in \Delta $. Next, for each fixed tuple $ \alpha =
( \alpha_2, \ldots, \alpha_m) \in \mathbf{B}^{n_2} \times \ldots
\times \mathbf{B}^{n_m} $ and $ 1 \leq j \leq n_1 $, define
\begin{eqnarray*}
\psi_{1,j}^{\alpha}(z) = \big( \phi_{1,j}(z), \alpha_2, \ldots,
\alpha_m \big)
\end{eqnarray*}
for $ z $ in $ \Delta $ where $ \phi_{1,j} $ is as described
above. Then
\[
\psi_{1,j}^{\alpha} : \Delta \rightarrow \mathbf{B}^{n_1} \times
\mathbf{B}^{n_2} \times \ldots \times \mathbf{B}^{n_m}
\]
is an extremal disc for the Kobayashi distance. Indeed, for all
$p, q$ in $ \Delta $,
\[
d_{\mathbf{B}^{n_1} \times \mathbf{B}^{n_2} \times \ldots \times
\mathbf{B}^{n_m}} \big( \psi_{1,j}^{\alpha}(p),
\psi_{1,j}^{\alpha}(q) \big) =
\]
\[
d_{\mathbf{B}^{n_1} \times \mathbf{B}^{n_2} \times \ldots \times
\mathbf{B}^{n_m}} \big( \phi_{1,j}(p), \alpha_2, \ldots, \alpha_m
), \phi_{1,j}(q), \alpha_2, \ldots, \alpha_m ) \big).
\]
Now, we turn to the explicit formulae for the Kobayashi distance
on product of balls -- so that the right hand side above equals
\[
\max  \Big \{ d_{\mathbf{B}^{n_1}} \big( \phi_{1,j} (p),
\phi_{1,j} (q)\big) , d_{\mathbf{B}^{n_2}}( \alpha_2, \alpha_2),
\ldots, d_{\mathbf{B}^{n_m}} (\alpha_m, \alpha_m) ) \Big \}
\]
which in turn equals
\[
d_{\mathbf{B}^{n_1}} \big( \phi_{1,j} (p), \phi_{1,j} (q)\big).
\]
Finally, from (\ref{D0}) we see that
\[
d_{\mathbf{B}^{n_1} \times \mathbf{B}^{n_2} \times \ldots \times
\mathbf{B}^{n_m}} \big( \psi_{1,j}^{\alpha} (p),
\psi_{1,j}^{\alpha}(q) \big) = d_{\Delta} (p,q).
\]

\medskip

\noindent Once we know that $ \psi_{1,j}^{\alpha} $ are complex
geodesics, consider the composition
\[
{f}_{1,j}^{\alpha} = f \circ \psi_{1,j}^{\alpha} : \Delta
\rightarrow \mathbf{B}^{n}
\]
for each fixed $ \alpha = ( \alpha_2, \ldots, \alpha_m) \in
\mathbf{B}^{n_2} \times \ldots \times \mathbf{B}^{n_m} $ and $ 1
\leq j \leq n_1 $ which is the restriction of $f$ to the line in $
\mathbf{B}^{n_1} \times \mathbf{B}^{n_2} \times \ldots \times
\mathbf{B}^{n_m} $ parameterized by $ \psi_{1,j}^{\alpha} $.
Evidently, $ {f}_{1,j}^{\alpha} $ preserves the Kobayashi
distance, i.e.,
\[
d_{\mathbf{B}^n} \big({f}_{1,j}^{\alpha}(p), {f}_{1,j}^{\alpha}(q)
\big) = d_{\Delta} (p,q)
\]
for all $ p, q $ in $ \Delta $.

\medskip

\noindent Using the arguments similar to those used in the proof
of Step I of Theorem \ref{0.1}, one can show that $
{f}_{1,j}^{\alpha} $ is $ C^{\infty}$-smooth. This requires the
fact that $ F^K_{\Delta} $ and $ F^K_{\mathbf{B}^n} $ are both
Riemannian (i.e., both are quadratic forms associated to
Riemannian metrics) and that $ {f}_{1,j}^{\alpha} $ is distance
preserving. The goal now is to show that $ {f}_{1,j}^{\alpha} $ is
holomorphic or anti-holomorphic for each fixed $j$ and $ \alpha $.
This follows by repeating the reasoning in Step II of Theorem
\ref{0.1}.

\medskip

\noindent Moreover, the connectedness of $ \mathbf{B}^{n_2} \times
\ldots \times \mathbf{B}^{n_m} $ and the set of all complex lines
together with $ C^1$-smoothness of $f$ (used in exactly the same
way as in Theorem \ref{0.1} and only here) forces that either $
{f}_{1,j}^{\alpha} $ is holomorphic for every choice $ \alpha \in
\mathbf{B}^{n_2} \times \ldots \times \mathbf{B}^{n_m} $ or
anti-holomorphic for every $ \alpha = ( \alpha_2, \ldots,
\alpha_m) \in \mathbf{B}^{n_2} \times \ldots \times
\mathbf{B}^{n_m} $ for each fixed $j $. As before, replacing $
{f}_{1,j}^{\alpha} $ by it's complex conjugate, if necessary, we
may assume that $ {f}_{1,j}^{\alpha} $ is holomorphic for each
fixed $ 1 \leq j \leq n_1 $ and for all $ \alpha $. Repeating this
procedure shows that $f$ restricted to all complex lines in $
\mathbf{B}^{n_1} \times \mathbf{B}^{n_2} \times \ldots \times
\mathbf{B}^{n_m} $ is holomorphic. This allows us to conclude that
$f$ is a biholomorphism from $ \mathbf{B}^{n_1} \times
\mathbf{B}^{n_2} \times \ldots \times \mathbf{B}^{n_m} $ onto $
\mathbf{B}^n $. This contradicts the fact that $ \mathbf{B}^n $
cannot be mapped biholomorphically onto any product domain thereby
finishing the proof for the Kobayashi metric. Furthermore, since
the Kobayashi and the Carath\'{e}odory metrics are equal on $
\mathbf{B}^n $ and $ \mathbf{B}^{n_1} \times \mathbf{B}^{n_2}
\times \ldots \times \mathbf{B}^{n_m} $, the result follows. \qed

\medskip

\noindent \textit{Proof of Theorem \ref{0.3}:} Suppose for some $
2 \leq m \leq n $, there exists a $C^1 $-Kobayashi isometry $f :
D_1 \times D_2 \times \ldots \times D_m \rightarrow \mathbf{B}^n
$. Now, fix $ a = (a_2, \ldots, a_{m}) \in D_2 \times D_2 \times
\ldots \times D_m $ and consider $ f_a : D_1 \rightarrow
\mathbf{B}^n $ defined by $ f_a(z) = f( z, a_2, \ldots, a_m) $ for
$ z $ in $ D_1 $. Note that for any $ z, w \in D_1 $,
\[
d_{\mathbf{B}^n}  \big( f_a(z), f_a(w) \big) = d_{\mathbf{B}^n}
\big( f( z, a_2, \ldots, a_m), f( w, a_2, \ldots, a_m)  \big)
\]
Since $f$ is a Kobayashi isometry, the right hand side above
equals
\[
d_{D_1 \times D_2 \times \ldots \times D_m } \big( ( z, a_2,
\ldots, a_m), ( w, a_2, \ldots, a_m)  \big)
\]
which in turn by the product formula of the Kobayashi metric is
given by
\[
\max \big \{ d_{D_1}(z,w), d_{D_2} (a_2, a_2), \ldots, d_{D_m}
(a_m, a_m) \big \}
\]
which equals
\[
d_{D_1}(z, w).
\]
The above calculation shows that $ d_{\mathbf{B}^n}  \big( f_a(z),
f_a(w) \big) = d_{D_1}(z, w) $ for all $ z, w $ in $D_1 $. Said
differently, the mapping $ f_a : D_1 \rightarrow \mathbf{B}^n $ is
distance preserving.

\medskip

\noindent \textbf{Step I:} By Lemma 3.3 of
\cite{Seshadri&Verma-2009}, it is known that $ d_{D_1} $ is
Lipschitz equivalent to the Euclidean distance on compact convex
subdomains of $ D_1 $. To verify this, observe that $ F^K_{D_1} $
is jointly continuous by virtue of the tautness of the domain $
D_1 $. Hence, $ F^K_{D_1} ( \cdot, v) \approx |v| $ on any compact
subset of $ D_1 $. Integrating the above estimate along straight
line segments and complex geodesics joining any two points $p,q
\in D_1 $, we get the required result. Note that convexity of $
D_1 $ guarantees the existence of geodesics between any two points
in $ D_1 $ and that the line segment joining these two points is
contained in $ D_1 $. Therefore, from the classical theorem of
Rademacher and Stepanov, we see that $ f_a $ is differentiable
almost everywhere.

\medskip

\noindent \textbf{Step II:} Firstly, it follows from
\cite{Lempert-1981} that $ F^K_{D_1} $ is $ C^1$-smooth on $ D_1
\times \mathbf{C}^{n_1} \setminus \{0\} $. Secondly, an argument
similar to that used in \cite{Seshadri&Verma-2009} yields that the
infinitesimal metric $ F^K_{D_1} $ is Riemannian. $ F^K_{D_1} $
being Riemannian at $ p \in D_1 $ is equivalent to the
`parallelogram law' being satisfied on $ T_p D_1 $, i.e.,
\begin{equation} \label{II}
\big( F^K_{D_1} (p, v + w) \big)^2 + \big( F^K_{D_1} (p, v -
w)\big)^2 = 2 \left( \big( F^K_{D_1} (p,v)\big)^2 +  \big(
F^K(p,w) \big)^2 \right)
\end{equation}
for all $ v, w \in  T_p D_1 $. This is verified by first showing
that $ F^K_{D_1} = f_a^*( F^K_{\mathbf{B}^n}) $ at every point of
differentiability of $f_a $ which in turn relies on
\cite{Myers&Steerod-1939} and existence of smooth geodesics in $
D_1$. Once we know that $ F^K_{D_1} $ is Riemannian at every point
of differentiability of $f_a $ which is a dense subset of $D_1$,
fix $v, w $ in (\ref{II}) and use the continuity of $ F^K_{D_1} $
in the domain variable to conclude.

\medskip

\noindent \textbf{Step III:} Since $ f_a $ is a continuous
distance preserving mapping between two $ C^1$ Riemannian
manifolds $ (D_1, F^K_{D_1} ) $ and $ (\mathbf{B}^n,
F^K_{\mathbf{B}^n }) $, applying the theorem of Myers-Steenrod
(\cite{Myers&Steerod-1939}) gives us that $ f_a $ is $ C^1 $.

\medskip

\noindent \textbf{Step IV:} $ f_a $ is
holomorphic/anti-holomorphic. This follows exactly as in Step II
of Theorem \ref{0.1}.

\medskip

\noindent By the connectedness of $ D_1$ and $ C^1$-smoothness of
the isometry $f$, an argument similar to the one used in Theorem
\ref{0.1} immediately shows that either $ f_a $ is holomorphic for
every $ a \in D_1 $ or conjugate holomorphic for every $ a $ in $
D_1 $. Applying complex conjugation, if necessary, we may assume
that $ f_a $ is holomorphic for every choice of $ a \in D_1 $.
Likewise, one can show that the mappings $ f_b $ given by $ f_b(z)
= f( b_1, z, b_3, \ldots, b_m), z \in D_2 $ are all holomorphic
for every parameter $ b= (b_1, b_3, \ldots, b_m ) \in D_1 \times
D_3 \times \ldots \times D_m $. Repeating this argument, we see
that $f$ is separately holomorphic with respect to a group of
variables for any fixed value of the other ones. In this setting,
a generalisation of the classical Hartog's theorem due to
Herv\'{e} (see Theorem 2 in section II.2.1 of \cite{Herve}) shows
that $f $ is holomorphic on $ D_1 \times D_2 \times \ldots \times
D_m $ and consequently $ D_1 \times D_2 \times \ldots \times D_m $
is biholomorphic to $ \mathbf{B}^n $. This contradiction finishes
the proof for the Kobayashi metric. Since the Kobayashi and the
Carath\'{e}odory metric are equal on bounded convex domains (cf.
\cite{Lempert-1982}), the theorem is completely proven. \qed

\section{Continuous extendability up to the
boundary of isometries of the Kobayashi metric}

\noindent In order to be able to prove Theorem \ref{1}, we need to
introduce the following special coordinates constructed for weakly
pseudoconvex finite type domains in \cite{Catlin-1989}:

\medskip

\noindent Let $ D \subset \mathbf{C}^2 $ be a domain whose
boundary is smooth pseudoconvex and of finite type $ 2m, m \in
\mathbf{N} $ near the origin. Let $ U$ be a tiny neighbourhood of
the origin and $ \rho $ a smooth defining function on $ U $ such
that $ U \cap \partial D = \{ \rho = 0 \} $ and $ \frac{ \partial
\rho}{\partial z_2} (0,0) \neq 0 $. Then for each $ \zeta \in U
\cap D $, there exists a unique automorphism $ \phi^{\zeta} $ of $
\mathbf{C}^2 $ defined by
\[
\phi^{\zeta}(z_1,z_2) = \Bigg( z_1 - \zeta_1, \Big (z_2 - \zeta_2
- \displaystyle\sum_{l=1}^{2m} d^l( \zeta) ( z_1 - \zeta_1 )^l
\Big) \left (d^0(\zeta)\right) ^{-1} \Bigg)
\]
where $ d^l(\zeta) $ are non-zero functions depending smoothly on
$ \zeta $ with the property that the function $ \rho \circ
(\phi^{\zeta})^{-1} $ satisfies
\[
\rho \circ (\phi^{\zeta})^{-1} (w_1,w_2) = 2 \Re w_2 +
\displaystyle\sum _{l=2} ^{2m} P_{l,\zeta} (w_1,\bar{w}_1) + o
\big( |w_1|^{2m} + \Im w_2 \big)
\]
where $ P_{l,\zeta} (w_1,\bar{w}_1)$ are real-valued homogeneous
polynomials of degree $l$ without any harmonic terms.

\medskip

\noindent Let $ \| \cdot \| $ be a fixed norm on the finite
dimensional space of all real-valued polynomials on the complex
plane with degree at most $2m$ that do not contain any harmonic
terms. Define for some small $ \delta > 0 $
\[
\tau ( \zeta, \delta ) = \min_{ 2 \leq l \leq 2m } \left (
\frac{\delta}{ \| P_{l,\zeta} (w_1,\bar{w}_1) \| } \right) ^{1/l}.
\]
Let $ \Delta_{\zeta}^{\delta}: \mathbf{C}^2 \rightarrow
\mathbf{C}^2 $ be anisotropic dilations defined by
\begin{equation*}
\Delta_{ \zeta}^{\delta}(z_1,z_2)  = \left( \frac{z_1} {\tau (
\zeta, \delta)}, \frac{z_2}{ \delta } \right).
\end{equation*}
A useful set for approximating the geometry of $D$ near the origin
is the Catlin's bidisc $Q (\zeta, \delta )$ determined by the
quantities $ \tau( \zeta, \delta) $ where
\begin{equation*}
Q (\zeta, \delta ) = \Big(\Delta_{ \zeta}^{\delta} \circ \phi
^{\zeta} \Big)^{-1}( \Delta \times \Delta ).
\end{equation*}

\medskip

\noindent The proof of Theorem \ref{1} also requires the following
estimates on the Kobayashi metric near a weakly pseudoconvex
boundary point of finite type.

\begin{prop} \label{S}
Let $D$ be a bounded domain in $ \mathbf{C}^2 $. Assume that $
\partial D $ is $ C^{\infty} $-smooth weakly pseudoconvex of
finite type near a point $ p^0 \in \partial D $. Given $ \epsilon
> 0 $, there exist positive numbers $ r_2 < r_1 < \epsilon $ and $
C$ such that the following inequality is true:
\[
d_D (a, b) \geq -(1/2) \log d( b, \partial D) - C, \ a \in D
\setminus B(p^0, r_1), \ b \in B(p^0, r_2) \cap D.
\]

\end{prop}

\begin{proof} By Theorem 1.1 of \cite{Berteloot-2003} there exists
a neighbourhood $U$ of $p^0$ in $ \mathbf{C}^2$ such that
\begin{equation} \label{S1}
F^K_{D}(z, v) \approx  \frac{|v_T|} { \tau \big( z, d(z, \partial
D) \big) } + \frac{ |v_N|} {d(z, \partial D)}
\end{equation}
for all $z \in U \cap D $ and $v$ a tangent vector at $z$. As
usual the decomposition $ v = v_T + v_N $ into the tangential and
normal components is taken at $ \pi(z) \in
\partial D $ which is the closest point on $ \partial D $ to $z$ and
$ \tau \big( z, d(z, \partial D) \big) $ is as described above.
Let $ \gamma $ be an arbitrary piecewise $ C^1$-smooth curve in $
D$ joining $a$ and $ b$, i.e., $ \gamma(0) = a, \gamma(1) = b $.
As we travel along $ \gamma$ starting from $ a$, there is a last
point $ \alpha $ on the curve with $ \alpha \in
\partial U \cap D$. Let $ \gamma(t) = \alpha$ and call $
\sigma $ the subcurve of $ \gamma $ with end-points $ b$ and $
\alpha $. Then $ \sigma $ is contained in a $ \delta
$-neighbourhood of $ \partial D$ for some fixed uniform $ \delta
> 0 $. Using (\ref{S1}) we get:
\begin{eqnarray*}
\int_0^1 F^K_{D} \big( \gamma(t), \dot{\gamma}(t) \big) dt & \geq
& \int_{t}^1 F^K_{D} \big( \sigma(t), \dot{\sigma}(t)
\big) dt \\
& \gtrsim &  \int_{t}^1 \frac{ | \dot{\sigma}_T(t) |} { \tau \big(
\sigma (t), d(\sigma (t),\partial D) \big) } dt +
\int_{t}^1 \frac{ | \dot{\sigma}_N(t) |} { d(\sigma(t),\partial D)  } dt \\
& \geq & \int_{t}^1 \frac{ | \dot{\sigma}_N(t) |} {
d(\sigma(t),\partial D)  } dt.
\end{eqnarray*}
The last integrand is seen to be at least
\[
\frac{d}{dt} \log \big( d( \sigma(t), \partial D) \big)^{1/2}
\]
(see for example Lemma 4.1 of \cite{Balogh&Bonk-2000}) and hence
\[
\int_0^1 F^K_{D} \big( \gamma(t), \dot{\gamma}(t) \big) dt \gtrsim
-(1/2) \log d(b, \partial D) - C
\]
for some uniform $ C> 0 $. Taking the infimum over all such $
\gamma$ it follows that
\[
d_{D} (a, b) \gtrsim -(1/2) \log d(b, \partial D) - C.
\]
\end{proof}

\begin{prop} \label{Q}
Let $D$ be a bounded domain in $ \mathbf{C}^2 $.
Assume that $ \partial D $ is $ {C}^{\infty}$-smooth weakly
pseudoconvex of finite type near two distinct boundary points $
a^0$ and $ b^0 $. Then for a suitable constant $ C$,
\[
d_D(a,b) \geq -(1/2) \log d(a, \partial D) -(1/2) \log d(b,
\partial D) - C
\]
whenever $ a, b \in D$, $ a $ is near $a^0 $ and $ b $ is near
$b^0 $.
\end{prop}

\begin{proof} Each path in $ D $ joining $a $ and $b$
must exit from neighbourhoods of $ a^0 $ and $ b^0 $. Hence the
result follows from Proposition \ref{S}.
\end{proof}

\medskip

\noindent \textit{Proof of Theorem \ref{1}:} Suppose that $f$ does
not extend continuously to any neighbourhood $ U_1 $ of $ p^0 $ in
$ \overline{D}_1 $. Then there exists a sequence of points $ \{
s^j \} \subset D_1 $ converging to $ p^0 \in \partial D_1 $ such
that the corresponding image sequence $ \{ f(s^j ) \} $ does not
converge to the point $ q^0 \in \partial D_2 $. Note that by
hypothesis, there exists a sequence $ \{ p^j \} \subset D_1 $
converging to $ p^0 \in \partial D_1 $ such that the limit $
\displaystyle \lim_{j \rightarrow \infty} f(p^j) = q^0 \in
\partial D_2 $ exists.

\medskip

\noindent Consider polygonal paths $ \gamma ^j $ in $ D_1 $
joining $ p^j $ and $ s^j $ defined as follows -- for each $j$,
choose $ p^{j0}, s^{j0} \in \partial D_1 $ closest to $ p^j $ and
$ s^j $ respectively. Set $ p^{j'} = p^j - |p^{j}- s^{j} |
n(p^{j0})$ and $ s^{j'}= s^j - |p^j - s^j| n(s^{j0})$ where $ n(z)
$ denotes the outward unit normal to $ \partial D_1 $ at $ z \in
\partial D_1 $. Let $ \gamma^j $ be the union of three segments:
the first one being the straight line path joining $p^{j}$ and
$p^{j'}$ along the inward normal to $ \partial D_1 $ at the point
$p^{j0}$, the second one being a straight line path joining
$p^{j'}$ and $ s^{j'}$ and finally the third path is taken to be
the straight line path joining $ s^{j'} $ and $ s^{j} $ along the
inward normal to the point $ s^{j0} $. Then $ f \circ \gamma ^j $
is a continuous path in $ D_2 $ joining $ f(p^j) $ and $ f(s^j) $.
Now, for each $ j $, pick $ u^j \in {B}( q^0, \epsilon) \cap U_2 $
on trace$(f \circ \gamma ^j)$ for some $ \epsilon > 0 $
sufficiently small. Let $ \{ t^j \} \subset D_1 $ be such that $
f(t^j) = u^j $. Then $ t^j \in \mbox{trace}(\gamma^j) $ and hence
$ t^j \rightarrow p^0 $ as $ j \rightarrow \infty $ by
construction. Moreover, the sequence $ f(t^j) = u^j \rightarrow
u^0 \in U_2 \cap \partial D_2, (u^0 \neq q^0) $. It follows from
\cite{Forstneric&Rosay-1987} that
\begin{multline} \label{Y1}
d_{D_1}(p^j,t^j)  \leq  - (1/2) \log {d(p^j,\partial D_1)} +
 (1/2) \log \big( d(p^j,\partial D_1) + |p^j - t^j| \big) \\
 + (1/2) \log \big( d(t^j,\partial D_1) + |p^j - t^j| \big) - (1/2) \log {d(t^j,\partial D_1)} +
 C_1
\end{multline}
and
\begin{equation} \label{Y2}
d_{D_2} \big( f(p^j), f(t^j) \big) \geq - (1/2) \log d \big(
f(p^j), \partial D_2 \big) - (1/2) \log d \big( f(t^j), \partial
D_2 ) - C_2
\end{equation}
for all $j$ large and uniform positive constants $ C_1 $ and $ C_2
$.

\medskip

\noindent \textbf{Assertion:} $ d \big( f(p^j), \partial D_2 )
\leq C_3 d( p^j, \partial D_2 ) $ and $ d \big( f(t^j), \partial
D_2 ) \leq C_3 d( t^j, \partial D_2 ) $ for some uniform positive
constant $ C_3 $.

\medskip

\noindent Grant this for now. Now, using the fact $ d_{D_1} ( p^j,
t^j) = d_{D_2} \big( f(p^j), f(t^j) \big) $ and comparing the
inequalities (\ref{Y1}) and (\ref{Y2}), it follows from the
assertion that for all $j$ large
\[
- ( C_1 + C_2 + \log C_3 ) \leq (1/2) \log \big( d(p^j,\partial
D_1) + |p^j - t^j| \big) + (1/2) \log \big( d(t^j,\partial D_1) +
|p^j - t^j| \big)
\]
which is impossible. This contradiction proves the theorem.

\medskip

\noindent It remains to establish the assertion. For this, fix $ a
\in D_1 $ and use Proposition \ref{S} to infer that
\begin{equation} \label{X1}
d_{D_1} (p^j, a) \geq -(1/2) \log d(p^j, \partial D_1) - C_4
\end{equation}
for some uniform positive constant $ C_4 $. On the other hand,
\begin{equation} \label{X2}
d_{D_2} \big( f(p^j), f(a) \big) \leq -(1/2) \log d \big(
d(f(p^j), \partial D_2 \big) + C_5.
\end{equation}
for all $j$ large and a uniform constant $ C_5 > 0 $. Fixing $ a $
in $ D_1 $, using $  d_{D_1} (a, p^j) = d_{D_2} \big( f(p^j), f(a)
\big) $, and comparing the inequalities (\ref{X1}) and (\ref{X2}),
we get the required estimates. Hence the assertion. Thus Theorem
\ref{1} is completely proven. \qed

\medskip

\noindent The proof of Theorem \ref{2} relies on the following
lemma.

\begin{lem} \label{R} Let $ D$ be a Kobayashi hyperbolic
domain in $ \mathbf{C}^n$ with a subdomain $ D' \subset D$. Let $
p,q \in D'$, $ d_{D}(p,q) = a$ and $ b > a$. If $ D'$ satisfies
the condition $ B_{D} (q,b) \subset D'$, then the following two
inequalities hold:
\begin{eqnarray*}
d_{D'} (p,q) \leq \frac{1}{\tanh (b-a) } d_{D}(p,q), \\
F^K_{D'} (p, v) \leq \frac{1}{\tanh (b-a) } F^K_{D}(p,v).
\end{eqnarray*}
\end{lem}

\noindent The reader is referred to \cite{Kim&Krantz-2008} (or
\cite{Kim&Ma-2003}) for a proof, but it should be noted that this
statement emphasizes an upper bound for $ d_{D'} $ in terms of $
d_D$. An estimate with the inequality reversed is an immediate
consequence of the definition of the Kobayashi metric.

\medskip

\noindent The second ingredient is an estimate for the Kobayashi
and the Carath\'{e}odory inner distance between two points in a
weakly pseudoconvex finite type domain $D$ in $ \mathbf{C}^2$ due
to Herbort (\cite{Herbort-2005}). To state this, let $ d(\cdot,
\partial D)$ be the Euclidean distance to the boundary and $ \rho
$ a smooth defining function for $ \partial D$. For $ a, b \in D$,
define
\begin{eqnarray*}
\rho^*(a,b) & = & \log \left ( 1 + \frac{ \tilde{d}(a,b)}{ d(a,
\partial D)} + \frac{| \langle L(a), a -b \rangle|}{ \tau ( a,
d(a, \partial D)) } \right) \\
L(a) & = & \left( - \frac{ \partial \rho} {\partial z_2}(a),
\frac{ \partial \rho} {\partial z_1}(a) \right) \\
d'(a,b) & = & \inf \big \{ \delta > 0 : a \in Q ( b, \delta) \big
\} \\
\tilde{d}(a,b) & = & \min \big \{ d'(a,b), |a-b| \big \},
\end{eqnarray*}
where $ \langle \cdot, \cdot \rangle $ denotes the standard
hermitian inner product in $ \mathbf{C}^2 $.
\medskip

\noindent The main result of \cite{Herbort-2005} that is needed
is:

\begin{thm} \label{X0}
Assume that $ D = \{ \rho < 0 \} \subset \mathbf{C}^2$ be a
bounded pseudoconvex domain with smooth boundary such that all
boundary points are of finite type. Then there exists a positive
constant $ C_* $ such that for any two points $ a, b \in D $
\[
C_* \big ( \rho^*(a,b) + \rho^*(b,a) \big) \leq c_D^i(a,b) \leq
d_D(a,b) \leq 1/C_* \big ( \rho^*(a,b) + \rho^*(b,a) \big).
\]
\end{thm}

\medskip

\noindent \textit{Proof of Theorem \ref{2}:} Suppose that there
exists a $ C^0$-isometry $ f : D_1 \rightarrow D_2 $ with the
property that: there exists a sequence $ \{p^j \} \subset D_1 $
converging to $ p^0 \in \partial D_1 $ such that the corresponding
image sequence $ \{ f(p^j) \} \subset D_2 $ converges to the point
$ q^0 \in
\partial D_2 $. The proof involves several steps.

\medskip

\noindent \textbf{Step I:} $f$ extends continuously to a
neighbourhood of $ p^0 $ in $ \overline{D}_1 $. This is immediate
from Theorem \ref{1}.

\medskip

\noindent \textbf{Step II:} Pick a sequence $ \{a^j \} \subset D_1
$ that converges normally to the origin, i.e, $ a^j = (0, -
\delta_j) $ where $ \delta_j > 0 $. It follows from Step I that
the corresponding image sequence $ b^j = f(a^j) \rightarrow q^0
\in \partial D_2 $ as $ j \rightarrow \infty $. It will be useful
to briefly describe the scaling of domains $D_1$, $ D_2$ and the
corresponding model domains in terms of the base point $p^0$ ($
q^0 $ respectively) and the sequence $ \{ a^j \} $ ($ \{ b^j \}$
respectively). These will require some basic facts about the local
geometry of the domain $ D_1 $ in a small neighbourhood $U_1 $ of
$ p^0 \in \partial D_1 $ and that of a strongly pseudoconvex
domain.

\medskip

\noindent \textbf{Scaling the domain $D_1$ with respect to $ \{a^j
\} $:}

\medskip

\noindent Let $ \Delta^j : \mathbf{C}^2 \rightarrow \mathbf{C}^2 $
be a sequence of dilations defined by
\begin{equation*}
\Delta^j (z_1,z_2)  = \left( \frac{z_1} {
 \delta_j ^{1/2m}}, \frac{z_2}{ \delta_j } \right).
\end{equation*}
Then the domains $ D_1 ^j = \Delta^j ( D_1)$ converge in the
Hausdorff metric to
\[
D_{1, \infty} = \Big \{ (z_1, z_2) \in \mathbf{C}^2 : 2 \Re z_2 +
|z_1|^{2m} < 0 \Big \}.
\]

\noindent We claim that $ d_{D_1^j} \big( (0,-1), \cdot \big)
\rightarrow d_{D_{1, \infty}} \big( (0,-1), \cdot \big) $
uniformly on compacts of $ D_{1, \infty} $. This was done in
\cite{Mahajan&Verma} and we include it here for completeness.
First, it is natural to prove convergence at the infinitesimal
level:

\begin{lem} \label{K1} For $(s,v) \in D_{1,\infty} \times
\mathbf{C}^2$,
\begin{equation*}
\displaystyle\lim_{j \rightarrow \infty} F^K_{D_1^j} (s,v) =
F^K_{D_{1,\infty}} (s,v). \label{7.1}
\end{equation*}
Moreover, the convergence is uniform on compact sets of $
D_{1,\infty} \times \mathbf{C}^2$.
\end{lem}

\begin{proof} Let $ S \subset D_{1,\infty}$ and $ G \subset \mathbf{C}^2$ be
compact and suppose that the desired convergence does not occur.
Then there is a $ \epsilon_0 > 0$ such that after passing to a
subsequence, if necessary, we may assume that there exists a
sequence of points $ \{ s^j \} \subset S $ which is relatively
compact in $ D_1^j$ and a sequence $ \{ v^j \} \subset G$ such
that
\begin{equation*}
\big| F^K_{ D_1^j }( s^j, v^j ) - F^K _{D_{1,\infty}}( s^j, v^j )
\big|
> \epsilon_0
\end{equation*}
for $j$ large. Additionally, $ s^j \rightarrow s \in S $ and $ v^j
\rightarrow v \in G $ as $ j \rightarrow \infty$. Since $
F^K_{D_{1,\infty}} ( s, \cdot) $ is homogeneous, we may assume
that $ | v^j| = 1 $ for all $j$. Observe that $ D_{1,\infty} $ is
complete hyperbolic and hence taut. The tautness of $
D_{1,\infty}$ implies via a normal family argument that $
F^K_{D_{1,\infty}} ( \cdot, \cdot) $ is jointly continuous, $ 0 <
F^K_{D_{1,\infty}} (s,v) < \infty $ and there exists a holomorphic
extremal disc $ g : \Delta \rightarrow D_{1,\infty} $ that by
definition satisfies $ g(0) = s, {g}'(0) = \mu v $ where $ \mu
> 0 $ and $ F^K_{D_{1,\infty}} (s,v)= 1/{\mu} $. Hence
\begin{equation}
\big| F^K_{ D _1^j }( s^j, v^j ) - F^K_{D_{1,\infty}} (s,v) \big|
> \epsilon_0 / 2     \label{k1}
\end{equation}
for $j$ sufficiently large. Fix $ \delta \in (0,1) $ and define
the holomorphic mappings $ g^j : \Delta \rightarrow \mathbf{C}^2 $
by
\[
g^j (z) = g \left( (1 - \delta) z \right) + (s^j - s) + \mu ( 1 -
\delta)z( v^j - v ).
\]
Since the image $ g \left(( 1 - \delta)\Delta \right)$ is
compactly contained in $ D_{1,\infty}$ and $ s^j \rightarrow s,
v^j \rightarrow v$ as $ j \rightarrow \infty$, it follows that $
g^j: \Delta \rightarrow D_1^j $ for $j$ large. Also, $ g^j(0) =
g(0) + s^j - s = s^j$ and $ (g^j)'(0) = ( 1 - \delta) g'(0) + \mu
(1 - \delta)(v^j - v) = \mu (1 - \delta) v^j $. By the definition
of the infinitesimal metric it follows that
\begin{equation*}
F^K_{ D_1^j }( s^j, v^j ) \leq \frac{1}{ \mu (1 - \delta) } =
\frac{F^K_{D_{1,\infty}} (s,v)}{ (1 - \delta) }.
\end{equation*}
Letting $ \delta \rightarrow 0^+ $ yields
\begin{equation}
\displaystyle \limsup_{j \rightarrow \infty}F^K_{ D_1^j }( s^j,
v^j ) \leq F^K_{D_{1,\infty}} (s,v). \label{4.4}
\end{equation}
Conversely, fix $\epsilon > 0$ arbitrarily small. By definition,
there are holomorphic mappings $ h^j: \Delta \rightarrow D_1^j$
satisfying $ h^j(0) = s^j$ and $ (h^j)'(0) = \mu^j $ where $ \mu^j
> 0 $ and
\begin{equation}
F^K_{ D_1^j }( s^j, v^j ) \geq \frac{1}{ \mu^j} - \epsilon
\label{4.5}
\end{equation}
The sequence $ \{ h^j \} $ has a subsequence that converges to a
holomorphic mapping $ h : \Delta \rightarrow D_{1,\infty}$
uniformly on compact sets of $ \Delta $. To see this, consider
$\Delta(0,r)$ for $r \in (0,1)$. Now, $ \phi^{p^0} = \phi^{(0,0)}
= id_{\mathbf{C}^2}$ and $ \tau( p^0, \delta_j) = \tau \big(
(0,0), \delta_j \big) \approx (\delta_j)^{1/2m} $. Further, we may
assume that $S$ is compactly contained in $ \Delta(0,C_1^{1/2m})
\times \Delta (0,C_1)$ for some $ C_1 > 1$. As a consequence
\[
(\Delta^j)^{-1}(s^j) \in Q (p^0, C_1 \delta_j).
\]
for all $j$. Also, note that
\[
( \Delta^j)^{-1} ( s^j ) \rightarrow p^0 \in \partial D_1
\]
as $ j \rightarrow \infty $. Now, applying Proposition $1$ in
\cite{Berteloot&Coeure-1991} to the mappings
\[
(\Delta^j)^{-1} \circ h^j : \Delta \rightarrow D_1
\]
shows that there exists a uniform positive constant $ C_2= C_2(r)
$ with the property that
\[
(\Delta^j )^{-1} \circ h^j \big(\Delta(0, r)\big) \subset Q \big
(p^0, C_2 C_1 \delta_j \big)
\]
or equivalently that
\[
h^j \big( \Delta(0,r)\big ) \subset \Delta \left(0, (C_1
C_2)^{1/2m} \right) \times \Delta(0, C_1 C_2 ).
\]
Therefore, $ \{ h^j \} $ is a normal family. Hence, the sequence $
\{ h^j \}$ has a subsequence that converges uniformly on compact
sets of $ \Delta $ to a holomorphic mapping $ h : \Delta
\rightarrow \mathbf{C}^2 $ or $ h \equiv \infty $. The latter
cannot be true since $ h(0) = s $. It remains to show that $ h :
\Delta \rightarrow D_{1,\infty}$. For this note that $ D_1^j $ are
defined in a neighbourhood of the origin by
\[
2 \delta_j \Re z_2 + \delta_j |z_1|^{2m} + o \big( \delta_j |
z_1|^{2m} + \delta_j \Im z_2 \big) < 0.
\]
 Thus, for $w \in \Delta(0,r)$ and $ r \in (0,1)$
\[
2 \Re \big( h^j_2 (w) \big) + |h^j_1(w) | ^{2m} +
\frac{1}{\delta_j} o \left( \delta_j | h^j_1(w)|^{2m} + \delta_j
\Im \big(h^j_2(w)\big) \right) < 0.
\]
Letting $ j \rightarrow \infty$ yields
\[
2 \Re \big(h_2(w)\big) + |h_1(w)|^{2m} \leq 0
\]
or equivalently that $ h ( \Delta (0,r)) \subset \overline{
D}_{1,\infty}$. Since $ r \in (0,1) $ was arbitrary, it follows
that $ h( \Delta ) \subset \overline{ D}_{1,\infty} $. Since $
h(0,0) = s $ the maximum principle forces that $ h : \Delta
\rightarrow D_{1,\infty} $. Note that
\[
h'(0) = \displaystyle \lim_{j \rightarrow \infty} (h^j)'(0) =
\displaystyle \lim_{j \rightarrow \infty} \mu ^j v^j= \mu v
\]
for some $ \mu > 0$. It follows from the definition of the
infinitesimal metric that
\[
F^K_{D_{1,\infty}} (s,v) \leq 1/ {\mu}.
\]
The above observation together with (\ref{4.5}) yields
\begin{equation}
\displaystyle \liminf_{j \rightarrow \infty}F^K_{ D_1^j }( s^j,
v^j ) \geq F^K_{D_{1,\infty}} (s,v). \label{4.6}
\end{equation}
Combining (\ref{4.4}) and (\ref{4.6}) shows that
\begin{equation*}
\displaystyle \lim_{j \rightarrow \infty}F^K_{ D_1^j }( s^j, v^j )
= F^K_{D_{1,\infty}} (s,v)
\end{equation*}
which contradicts the assumption (\ref{k1}) and proves the lemma.
\end{proof}

\noindent To control the integrated Kobayashi distance on domains
$ D_1^j $, we first note the following:

\begin{lem} \label{M} Let $ D\subset \mathbf{C}^n $ be a bounded domain
and $ p^0 \in \partial D $ be a local holomorphic peak point. Then
for any fixed $ R > 0 $ and every neighbourhood $ U $ of $ p^0 $
there exists a neighbourhood $ V \subset U $ of $ p^0 $ with $ V$
relatively compact in $ U $ such that for all $ z \in V \cap D $,
we have
\[
B_{D}(z, c R ) \subset B_{U \cap D} (z, R) \subset B_{D}(z, R)
\]
where $ c > 0 $ is a constant independent of $ z \in V \cap D $.
\end{lem}

\begin{proof} Let $U$ be a  neighbourhood of $p^0$ and $ g \in
\mathcal{A}(U \cap D)$, the algebra of continuous functions on the
closure of $U \cap D$ that are holomorphic on $ U \cap D$, such
that $ g(p^0) = 1$ and $ | g(p)| < 1 $ for $ p \in \overline{ U
\cap D} \setminus \{p^0 \} $. Fix $ \epsilon > 0$. Then there
exists a neighbourhood $ V_1 \subset U$ of $ p^0$ such that
\begin{equation*}
 F^K_{D}(z, v ) \leq F^K_{U \cap D} (z,v) \leq ( 1 +
\epsilon) F^K_{D}(z, v)
\end{equation*}
for $ z \in V_1 \cap D $ and $v$ a tangent vector at $ z$. This is
possible by the localisation property of the Kobayashi metric (see
for example Lemma 2 in \cite{Royden-1971} or \cite{Graham-1975}).

\medskip

\noindent The first inequality evidently implies that $ B_{U \cap
D} (z, R) \subset B_{D}(z, R) $ for all $ z  \in V_1 \cap D $ and
all $R > 0$. For the lower estimate the following observation will
be needed. For every $ R > 0$ there is a neighbourhood $ V \subset
V_1 $ of $ p^0$ with the property that if $ z \in V \cap D $ then
$ B_{U \cap D} (z, R) \subset V_1 \cap D$. For this it suffices to
show that
\[
\displaystyle\lim_{z \rightarrow p^0 } d_{ U \cap D } ( z, ( U
\cap D) \setminus \overline{V_1 \cap D} ) = + \infty.
\]
Indeed for every $ p \in ( U \cap D) \setminus \overline{V_1 \cap
D}$,
\[
d_{ U \cap D} (z, p) \geq d_{\Delta} ( g(z), g(p)) \rightarrow +
\infty
\]
as $ z \rightarrow p^0 $ since $ g(p^0) = 1 $ and $ |g| < 1 $ on $
( U \cap D) \setminus \overline{V_1 \cap D}$. This proves the
claim.

\medskip

\noindent Now for a given $ R > 0 $ let $ V $ be a sufficiently
small neighbourhood of $ p^0$ so that
\[
B_{ U \cap D} (z, R) \subset V_1 \cap D
\]
if $ z \in V \cap D $. Pick $ p \in D$ in the complement of the
closure of $ B_{U \cap D} (z, R) $ and let $ \gamma : [0,1]
\rightarrow D $ be a differentiable path with $ \gamma(0) = z $
and $ \gamma (1) = p$. Then there is a $ t_0 \in (0,1) $ such that
$ \gamma( [ 0, t_0 ) ) \subset B_{U \cap D} (z, R) $ and $ \gamma
(t_0 ) \in
\partial B_{U \cap D} (z, R) $. Hence
\begin{eqnarray*}
\int_0^1 { F^K_{D} ( \gamma(t), \dot{\gamma}(t)) dt } & \geq &
\int_0^{t_0} { F^K_{D} ( \gamma(t), \dot{\gamma}(t)) dt } \\
& \geq & 1/ ( 1 + \epsilon ) \int_0^{t_0} { F^K_{U \cap D} (
\gamma(t), \dot{\gamma}(t)) dt } \\
& \geq & 1/ ( 1 + \epsilon ) \  d_{ U \cap D} ( z, \gamma(t_0) ) =
R / ( 1 + \epsilon )
\end{eqnarray*}
which implies that $ d_{D}(z, p) \geq R/ ( 1 + \epsilon ) $. In
other words,
\begin{equation*}
B_{D} \big(z, R/ (2( 1 + \epsilon) ) \big) \subset B_ { U \cap D}
(z, R) \label{Q2}
\end{equation*}
if $ z \in V \cap D $. Finally observe that
\[
B_{D} \big(z, R/ (2( 1 + \epsilon) ) \big) \subset B_ { U \cap D}
(z, R)  \subset B_D(z, R)
\]
for all $ z \in V \cap D $.
\end{proof}

\begin{lem} \label{O}
For all $ R>0 $ and for all $j$ large, $ B_{D_1^j} \big( (0,-1), R
\big) $ is compactly contained in $ D_{1, \infty} $.
\end{lem}

\begin{proof} First note that
\[
B_{D_1^j} \big( (0,-1), R \big) = \Delta^j \big(B_{D_1}( a^j, R)
\big).
\]
Since $ p^0 \in \partial D_1 $ is a local holomorphic peak point,
by lemma \ref{M}, we see that there exists a neighbourhood $ V
\subset U_1 $ of $ p^0 $ with $ V$ relatively compact in $ U_1 $
and a uniform positive constant $c$ such that for all $ z \in V
\cap D_1 $,
\[
B_{D_1}(z, c R ) \subset B_{U_1 \cap D_1} (z, R) \subset
B_{D_1}(z, R)
\]
and therefore it will suffice to show that $ \Delta^j \big(B_{U_1
\cap D_1}( a^j, R) \big) $ is compactly contained in $ D_{1,
\infty} $. The proof now divides into two parts. In the first part
we show that the sets $ \Delta^j \big(B_{U_1 \cap D_1}( a^j, R)
\big) $ cannot accumulate at the point at infinity in $
\partial D_{1,\infty} $ and in the second part we show that the
sets $ B_{D_1^j} \big( (0,-1), R \big) $ do not cluster at any
finite boundary point. Assume that $ p \in B_{U_1 \cap D_1}( a^j,
R)$. Using Herbort's lower estimate for the Kobayashi metric gives
us
\begin{equation*}
C_* \left( \rho^* (a^j,p) + \rho^* (p, a^j)\right) \leq d_{U_1
\cap D_1} (a^j,p).
\end{equation*}
As a consequence
\[
\tilde{d}(a^j,p) < \exp( R/ { C_*} )d(a^j, \partial D_1)
\]
which in turn implies that

\begin{itemize}

\item either $ |a^j - p| < d( a^j, \partial D_1) \exp( R/{C_*}) $
or

\medskip

\item for each $j$, there exists a $ \delta_j \in ( 0, d( a^j,
\partial D_1)\exp( R/{C_*}) ) $ such that $ a^j \in Q (p, \delta_j) $.

\end{itemize}
It follows from Proposition 1.7 in \cite{Catlin-1989} that there
exists a uniform positive constant $C$ such that for each $j$, the
following holds: if $ a^j \in Q (p, \delta_j)$, then $ p \in Q
\big( a^j, C \delta_j \big) $. Hence, the second statement above
can be rewritten as: there exists a positive constant $C$ such
that for each $j$, there exists a $ \delta_j \in ( 0, d( a^j,
\partial D_1)\exp( R/{C_*}) ) $ with the property that
\[
p \in (\phi^{a^j})^{-1} \Big ( \Delta(0, \tau( a^j, C \delta_j) )
\times \Delta(0, C \delta_j ) \Big).
\]
Said differently, $ B_{U_1 \cap D_1}( a^j, R) $ is contained in
the union
\begin{eqnarray*}
B_{U_1 \cap D_1}( a^j, R) \subset  B \Big( a^j, d(a^j,
\partial D_1) \exp{( R/ C_*)} \Big) \cup (\phi^{a^j})^{-1} \Big (
\Delta(0, \tau( a^j, C \delta_j) ) \times \Delta(0, C \delta_j )
\Big)
\end{eqnarray*}
with $ \delta_j $ as described above. Now,
\[
\Delta^j \left \lbrace (z_1,z_2) \in \mathbf{C}^2: | z_1 - a^j_1 |
^2 + | z_2 - a^j_2 | ^2 < \big(d( a^j, \partial D_1)\big)^2
\exp(2R/{C_*}) \right \rbrace =
\]
\begin{equation}  \label{O7}
\left \lbrace (w_1,w_2): |w_1|^2 +
\frac{\delta_j^2}{\delta_j^{1/m}} \left| w_2 + 1 \right|^2 <
\frac{\big (d( a^j, \partial D_1) \big)^2
\exp(2R/{C_*})}{\delta_j^{1/m}} \right \rbrace.
\end{equation}
If $ w= (w_1,w_2) $ belongs to the set described above, then
\begin{eqnarray}
|w_1|  & \leq &  \frac{ d( a^j, \partial D_1
)\exp(R/{C_*})}{\delta_j^{1/2m}} =  \frac{
\delta_j \exp(R/{C_*})}{\delta_j^{1/2m}} \qquad \mbox{and} \label{O1} \\
\left|  w_2 + 1 \right| & \leq  & \frac{ d( a^j, \partial D_1)
\exp(R/{C_*)} } { \delta_j } = \exp(R/{C_*)}. \label{O2}
\end{eqnarray}
Moreover, for $\delta_j \in \big( 0, d( a^j, \partial D_1) \exp(
R/{C_*}) \big)$,
\[
(\phi^{a^j})^{-1} \Big \lbrace (z_1,z_2) \in \mathbf{C}^2 : |z_1|
< \tau( a^j, C \delta_j) ,|z_2| <  C \delta_j  \Big \rbrace =
\]
\begin{equation*}
\left \lbrace (w_1,w_2): |w_1 - a^j_1| < \tau( a^j, C \delta_j),
\left| w_2 - a^j_2 - \displaystyle\sum_{l=1}^{2m} d^l(a^j) ( w_1 -
a^j_1)^l \right| < C \delta_j d^0(a^j) \right \rbrace
\end{equation*}
so that
\[
\Delta^j \circ (\phi^{a^j})^{-1} \Big (\Delta(0, \tau( a^j, C
\delta_j) ) \times \Delta(0, C \delta_j ) \Big) =
\]
\begin{equation}
\left \lbrace w: |w_1| < \frac{ \tau( a^j, C \delta_j)}{
\delta_j^{1/2m}}, \left| w_2 + 1 + {\delta_j}^{-1} \Big(
\displaystyle\sum_{l=1}^{2m} d^l(a^j) \delta_j^{l/2m} w_1^l \Big)
\right| < C d^0(a^j) \right \rbrace \label{O3}
\end{equation}
If $ w = (w_1,w_2) $ belongs to the set given by (\ref{O3}), then
\begin{eqnarray}
|w_1| & < & \frac{ \tau( a^j, C \delta_j)}{ \delta_j^{1/2m}} \qquad \mbox{and} \label{O4} \\
\left| w_2 + 1 + {\delta_j}^{-1}  \Big(
\displaystyle\sum_{l=1}^{2m} d^l(a^j) \delta_j^{l/2m} w_1^l \Big)
\right| & < & C  d^0(a^j). \label{O5}
\end{eqnarray}

\noindent Among other things, it was shown in \cite{Catlin-1989}
that

\begin{itemize}

\item $ {\delta_j}^{1/2} \lesssim \tau( a^j, \delta_j) \lesssim
{\delta_j}^{1/2m}$

\medskip

\item $ | d^l(a^j)| \lesssim \delta_j ( \tau(a^j, \delta_j) )^{-l}
$ for all $ 1 \leq l \leq 2m $

\medskip

\item $ d^0(a^j) \approx 1 $.

\end{itemize}

\noindent These estimates together with (\ref{O1}), (\ref{O2}),
(\ref{O4}) and (\ref{O5}) show that if $ w = (w_1, w_2) $ belongs
to either of (\ref{O7}) or (\ref{O3}), then $ |w| $ is uniformly
bounded. In other words, the sets
\[
\Delta^j \left( B( a^j, d(a^j, \partial D_1) \exp(R/ C_*)) \right)
\bigcup \Delta^j \left( (\phi^{a^j})^{-1} \left ( \Delta(0, \tau(
a^j, C \delta_j) ) \times \Delta(0, C \delta_j ) \right) \right)
\]
are uniformly bounded. Therefore, $ \Delta^j \big(B_{U_1 \cap
D_1}( a^j, R) \big)  $ and consequently $B_{D_1^j} \big( (0,-1), R
\big)$ as a set cannot cluster at the point at infinity on $
\partial D_{1,\infty}$.

\medskip

\noindent It remains to show that the sets $B_{D_1^j} \big(
(0,-1), R \big) $ do not cluster at any finite point  of $
\partial D_{1,\infty} $. Suppose there exists a sequence of points
$ \{ z^j \}, z^j \in B_{D_1^j} \big( (0,-1), R \big)$ such that $
z^j \rightarrow z^0 $ as $ j \rightarrow \infty $ where $ z^0 $ is
a finite point on $ \partial D_{1,\infty}$. Applying Theorem 1.1
of \cite{Berteloot-2003}, we see that there exists a neighbourhood
$ U $ of $ z^0 $ in $ \mathbf{C}^2 $ such that
\begin{equation} \label{O6}
 F^K_{D_1^j}(z, v) \approx  \frac{|v_T|} { \tau \big( z, d(z, \partial D_1^j) \big) } + \frac{ |v_N|} {d(z, \partial D_1^j)}
\end{equation}
uniformly for all $j$ large, $z \in U \cap D_{1, \infty} $ and $v$
a tangent vector at $z$ -- this stable version holds since the
defining functions for $ D_1^j $ converge to that of $
D_{1,\infty} $ in the $ C^{\infty}$-topology on a given compact
set. Here the decomposition $ v = v_T + v_N $ into the tangential
and normal components is taken at $ \pi^j(z) \in
\partial D_1^j $ which is closest point on $
\partial D_1^j $ to $z$. Note that
\[
d(z, \partial D_1^j) \approx d(z, \partial D_{1, \infty} )
\]
for $ z \in U \cap D_{1, \infty} $ and that $ \pi ^j(z)
\rightarrow \pi(z) \in \partial D_{1, \infty} $ where $ |\pi(z) -
z | = d( z, \partial D_{1, \infty} ) $. Let $ \gamma^j$ be an
arbitrary piecewise $ C^1$-smooth curve in $ D_1^j$ joining $z^j$
and $ (0,-1)$, i.e., $ \gamma^j(0) = (0,-1), \gamma^j(1) = z^j $.
As we travel along $ \gamma^j$ starting from $ (0,-1)$, there is a
last point $ \alpha^j$ on the curve with $ \alpha^j \in \partial U
\cap D_1^j$. Let $ \gamma^j(t_j) = \alpha^j$ and call $ \sigma^j$
the subcurve of $ \gamma^j$ with end-points $ z^j$ and $ \alpha^j
$. Then $ \sigma^j $ is contained in an $ \epsilon $-neighbourhood
of $ \partial D_1^j$ for some fixed uniform $ \epsilon > 0 $ and
for all $j$ large. Arguing as in the proof of Proposition \ref{S}
and using (\ref{O6}) we get:
\begin{eqnarray*}
\int_0^1 F^K_{D_1^j} \big( \gamma^j(t), \dot{\gamma}^j(t) \big) dt
& \geq & \int_{t_j}^1 F^K_{D_1^j} \big( \sigma^j(t),
\dot{\sigma}^j(t)
\big) dt \\
& \gtrsim &  \int_{t_j}^1 \frac{ | \dot{\sigma}^j_T(t) |} { \tau
\big( \sigma^j(t), d(\sigma^j(t),\partial D_1^j) \big) } dt +
\int_{t_j}^1 \frac{ | \dot{\sigma}^j_N(t) |} { d(\sigma^j(t),\partial D_1^j)  } dt \\
& \geq & \int_{t_j}^1 \frac{ | \dot{\sigma}^j_N(t) |} {
d(\sigma^j(t),\partial D_1^j)  } dt.
\end{eqnarray*}
As before the last integrand turns out to be at least
\[
\frac{d}{dt} \log \big( d( \sigma^j(t), \partial D_1^j)
\big)^{1/2}
\]
and consequently
\[
\int_0^1 F^K_{D_1^j} \big( \gamma^j(t), \dot{\gamma}^j(t) \big) dt
\gtrsim -(1/2) \log d(z^j, \partial D_1^j) + C
\]
for some uniform $ C> 0 $. Taking the infimum over all such $
\gamma^j$ it follows that
\[
d_{D_1^j} \big( z^j, (0,-1) \big) \gtrsim -(1/2) \log d(z^j,
\partial D_1^j) + C.
\]
This is however a contradiction since the left side is at most $R$
while the right side becomes unbounded. This completes the proof
of the lemma.
\end{proof}


\begin{prop} \label{K3}
\begin{equation*}
\displaystyle\lim_{j \rightarrow \infty} d_{D_1^j} \big( (0,-1),
\cdot \big)= d_{D_{1,\infty}} \big( (0,-1), \cdot \big).
\end{equation*}
Moreover, the convergence is uniform on compact sets of $
D_{1,\infty}$.
\end{prop}

\begin{proof} Let $ K $ be a compact subdomain of $D_{1,\infty}$ and suppose
that the desired convergence does not occur. Then there exists a $
\epsilon_0 > 0 $ and a sequence of points $ \{z^j \} \subset K $
which is relatively compact in $ D_1^j $ for all $j$ large such
that
\[
\big| d_{D_1 ^j} \big( (0,-1), z^j ) - d_{D_{1,\infty}} \big(
(0,-1), z^j ) \big|
> \epsilon _0.
\]
By passing to a subsequence, we may assume that $ z^j \rightarrow
z^0 \in K$ as $ j \rightarrow \infty $. Then using the continuity
of $ d_{D_{1,\infty}}(z^0, \cdot) $ we have
\[
\big| d_{D_1 ^j}  \big( (0,-1), z^j ) - d_{D_{1,\infty}} \big(
(0,-1), z^0 ) \big|
> \epsilon _0/2
\]
for all $j$ large. Fix $ \epsilon > 0$ and let $ \gamma : [0,1]
\rightarrow D_{1,\infty} $ be a path such that $ \gamma(0) =
(0,-1), \gamma(1) = z^0$ and
\[
\int_0 ^1 { F^K _{D_{1,\infty}} \big( \gamma(t), \dot{\gamma}(t)
\big)} dt < d_{D_{1,\infty}} \big( (0,-1), z^0 \big) + \epsilon/2.
\]
Define $ \gamma ^j : [0,1] \rightarrow \mathbf{C}^2 $ by
\[
\gamma^j(t) = \gamma(t) + ( z^j - z^0 ) t.
\]
Since the image $ \gamma([0,1])$ is compactly contained in
$D_{1,\infty}$ and $ z^j \rightarrow z^0 \in K $ as $ j
\rightarrow \infty$, it follows that $ \gamma ^j : [0,1]
\rightarrow D_1^j $ for $j$ large. In addition, $ \gamma^j (0) =
\gamma (0) = (0,-1) $ and $ \gamma^j (1) = \gamma(1) + z^j - z^0 =
z^j $. By Lemma \ref{K1}, we see that $ F^K_{D_1^j} ( \cdot,
\cdot) \rightarrow F^K_{D_{1,\infty}} (\cdot, \cdot) $ uniformly
on compact sets of $ D_{1,\infty} \times \mathbf{C}^2$. Also, note
that $ \gamma^j \rightarrow \gamma $ and $ \dot{\gamma}^j
\rightarrow \dot{\gamma} $ uniformly on $ [0,1]$. Therefore for
$j$ large, we obtain
\[
\int_0 ^1 { F^K_{D_1^j} \big( \gamma^j(t), \dot{\gamma}^j(t)
\big)} dt \leq \int_0 ^1 { F^K _{D_{1,\infty}} \big( \gamma(t),
\dot{\gamma}(t) \big)} dt + \epsilon/2 < d_{D_{1,\infty}} \big(
(0,-1), z^0 \big) + \epsilon.
\]
By definition of $ d_{D_1 ^j} \big( (0,-1), z^j \big)$ it follows
that
\[
d_{D_1 ^j} \big( (0,-1), z^j \big) \leq \int_0 ^1 { F^K_{D_1^j}
\big( \gamma^j(t), \dot{\gamma}^j(t) \big)} dt \leq
d_{D_{1,\infty}} \big( (0,-1), z^0 \big) + \epsilon.
\]
Thus
\begin{eqnarray}
\displaystyle\limsup_{j \rightarrow \infty} d_{D_1 ^j} \big(
(0,-1), z^j \big) \leq d_{D_{1,\infty}} \big( (0,-1), z^0 \big).
\label{N1}
\end{eqnarray}
To establish lower semi-continuity, we intend to use Lemma
\ref{R}. First note that the upper semi-continuity of the
integrated Kobayashi distance yields
\[
B_{ D_{1,\infty} } \big( (0,-1), R - \epsilon \big) \subset
B_{D_1^j} \big( (0,-1), R \big)
\]
for all $ R > 0$ and for all $j$ large. The Kobayashi completeness
of $ D_{1,\infty} $ implies that
\[
D_{1,\infty} = \displaystyle \bigcup _{\nu=1}^{\infty}
B_{D_{1,\infty}} \big((0,-1),\nu \big),
\]
i.e., $ D_{1,\infty} $ can be exhausted by an increasing union of
relatively compact subdomains $B_{D_{1,\infty}} \big((0,-1), \nu
\big) $. As a result, there exist uniform positive constants $
\nu^0 $ and $ \tilde{R} $ depending only on $K$ such that
\[
K \subset  B_{D_{1,\infty}} \big((0,-1), \nu^0 \big) \subset
B_{D_1^j} \big( (0,-1), \tilde{R} \big)
\]
for all $j$ large. By Lemma \ref{O}
\begin{equation*}
d_{D_{1,\infty}} \big( (0,-1), z^j \big)  \leq d_{B_{D_1^j} \big(
(0,-1), R' \big)} \big( (0,-1), z^j \big)
\end{equation*}
where  $ R' > 0 $ is chosen such that $ R' \gg 2 \tilde{R} $. Now,
apply Lemma \ref{R} to the domain $ D_1^j$. Let the Kobayashi
metric ball $ B_{D_1^j} \big((0,-1), R' \big)$ play the role of
the subdomain $ D'$. Then
\begin{equation*}
d_{B_{D_1^j} \big((0,-1), R' \big)} \big( (0,-1), z^j \big) \leq
\frac{d_{D_1^j} \big( (0,-1), z^j \big)} { \tanh \big( R'/2 -
d_{D_1^j} \big( (0,-1), z^j \big) \big)}.
\end{equation*}
Since $ z^j \in B_{D_1^j} \big( (0,-1), \tilde{R} ) $ for all $j$
large and the function $ x \rightarrow \tanh x$ is increasing on $
[0,\infty)$, it follows that
\begin{equation*}
d_{D_{1, \infty}} \big( (0,-1), z^j \big) \leq \frac{d_{D_1^j}
\big( (0,-1), z^j \big))} { \tanh \left( R'/2 - \tilde{R} \right)
}.
\end{equation*}
Letting $ R' \rightarrow \infty$ yields
\begin{equation*}
d_{D_{1,\infty}}\big( (0,-1), z^j \big) \leq \frac {d_{D_1^j}
\big( (0,-1), z^j \big)} { 1 - \epsilon }
\end{equation*}
for all $j$ large. Again exploiting the continuity of $
d_{D_{1,\infty}}(\cdot, \cdot)$ and (\ref{N1}), we see that
\begin{equation}
d_{D_{1, \infty}} \big( (0,-1), z^0 \big) \leq d_{D_1^j} \big(
(0,-1), z^j \big) + C \epsilon \label{N2}
\end{equation}
for all $j$ large. Combining the estimates (\ref{N1}) and
(\ref{N2}), we get
\[
\displaystyle\lim_{j \rightarrow \infty} d_{D_1 ^j} \big( (0,-1),
z^j \big) = d_{D_{1,\infty}}\big( (0,-1), z^0 \big).
\]
This is a contradiction and hence the result follows.
\end{proof}

\noindent \textbf{Scaling the domain $D_2$ with respect to $ \{b^j
\} $:}

\medskip

\noindent The following lemma in \cite{Pinchuk-1980} will be
useful in our situation.

\begin{lem} \label{F1} Let $ D$ be a strongly pseudoconvex domain,
$\rho$ a defining function for $ \partial D $ and $ p \in
\partial D $. Then there exists a neighbourhood $U$ of $p$
and a family of biholomorphic mappings $ h_{\zeta} : \mathbf{C}^n
\rightarrow \mathbf{C}^n$ depending continuously on $ \zeta \in U
\cap \partial D $ that satisfy the following:

\begin{enumerate}

\item [(i)]$ h_{\zeta}(\zeta) = 0 $.

\item [(ii)] The defining function $ \rho_{\zeta} = \rho \circ
h_{\zeta}^{-1} $ of the domain $ D ^{\zeta} := h_{\zeta} ( D) $
has the form
\[ \rho_{\zeta}(z) = 2 \Re \big( z_n + K_{\zeta}(z)\big) + H_{\zeta}(z) +
\alpha_{\zeta}(z) \] where $ K_{\zeta}(z) =
\displaystyle\sum_{i,j=1} ^n a_{ij} (\zeta) z_i z_j, H_{\zeta}(z)
= \displaystyle\sum_{i,j=1} ^n b_{ij} (\zeta) z_i \bar{z_j}$ and $
\alpha_{\zeta}(z)= o (|z|^2) $ with $ K_{\zeta}('z,0) \equiv 0 $
and $ H_{\zeta}('z,0) \equiv |'z|^2 $.

\item [(iii)] The mapping $ h_{\zeta}$ takes the real normal to $
\partial D $ at $ \zeta$ to the real normal $ \{ 'z =
y_n = 0 \} $ to $ \partial D ^{\zeta}$ at the origin.

\end{enumerate}

\end{lem}

\noindent Here, $ z \in \mathbf{C}^n$ is written as $ z = ( 'z,
z_n) \in \mathbf{C}^{n-1} \times \mathbf{C} $.

\medskip

\noindent To apply this lemma, choose points $ \zeta^j \in
\partial D_2$, closest to $ b^j $. For $j$ large, the choice of $
\zeta^j $ is unique since $ \partial D_2 $ is sufficiently smooth.
Moreover, $ \zeta^j \rightarrow q^0 $ and $ b^j \rightarrow q^0$
as $ j \rightarrow \infty $. Let $ h^j := h_{\zeta^j} $ be the
biholomorphisms provided by the lemma above. We observe that for
$j$ large, $ h^j(b^j) = ( 0, - \epsilon_j ) $. Let $ T^j :
\mathbf{C}^n \rightarrow \mathbf{C}^n $ be the anisotropic
dilation map given by
\[
T^j ( z_1, z_2 ) = \left( \frac{z_1}{ {\epsilon_j}^{1/2} } ,
\frac{z_2} { \epsilon_j} \right)
\]
and let $ D_2^j = T^j \circ h^j ( D_2 ) $.  Note that $ T^j \circ
h^j ( b^j ) = ( 0, -1) $ and the sequence of domains $ \{ D_2^j \}
$ converges in the Hausdorff metric to the unbounded realization
of the unit ball, namely to
\[
 D_{2, \infty} = \Big \{ z = (z_1,z_2) \in \mathbf{C}^2 : 2 \Re z_2 + | z_1 |^2 < 0
\Big \}.
\]
\noindent It is natural to investigate the behaviour of $
d_{D_2^j} ( z, \cdot )$ as $ j \rightarrow \infty$. To do this, we
use ideas from \cite{Seshadri&Verma-2006}.

\begin{prop} \label{D1} Let $ x^0 \in D_{2,\infty}$. Then $ d_{D_2^j} ( x^0,
\cdot ) \rightarrow d_{D_{2,\infty}} (x^0, \cdot ) $ uniformly on
compact sets of $ D_{2,\infty}$.
\end{prop}

\begin{proof} Let $ K \subset D_{2,\infty} $ be compact and suppose that the
desired convergence does not occur. Then there exists a $ \epsilon
_0 > 0 $ and a sequence of points $ \{z^j \} \subset K $ which is
relatively compact in $D_2^j $ for $j$ large such that
\[
\big| d_{D_2^j} ( x^0, z^j ) - d_{D_{2,\infty}} (x^0, z^j ) \big|
> \epsilon _0
\]
for all $j$ large. By passing to a subsequence, assume that  $ z^j
\rightarrow z^0 \in K $ as $ j \rightarrow \infty$. Since $
d_{D_{2,\infty}} ( x^0, \cdot ) $ is continuous, it follows that
\begin{equation}
\big| d_{D_2 ^j} ( x^0, z^j ) - d_{D_{2,\infty}} (x^0, z^0 ) \big|
> \epsilon _0/2 \label{Q3}
\end{equation}
for all $j$ large. The upper semicontinuity of the distance
function follows exactly as in Proposition \ref{K3}. Fix $
\epsilon
> 0$ and let $ \gamma : [0,1] \rightarrow D_{2,\infty} $ be a path
such that $ \gamma(0) = x^0, \gamma(1) = z^0$ and
\[
\int_0 ^1 { F^K _{D_{2,\infty}} \big( \gamma(t), \dot{\gamma}(t)
\big)} dt < d_{D_{2,\infty}} (x^0, z^0 ) + \epsilon/2.
\]
Define $ \gamma ^j : [0,1] \rightarrow \mathbf{C}^2 $ by
\[
\gamma^j(t) = \gamma(t) + ( z^j - z^0 ) t.
\]
Since the image $ \gamma([0,1])$ is compactly contained in
$D_{2,\infty}$ and $ z^j \rightarrow z^0 \in K $ as $ j
\rightarrow \infty$, it follows that $ \gamma ^j : [0,1]
\rightarrow D_2^j $ for $j$ large. In addition, $ \gamma^j (0) =
\gamma (0) = x^0 $ and $ \gamma^j (1) = \gamma(1) + z^j - z^0 =
z^j $. It is already known that $ F^K_{D_2^j} ( \cdot, \cdot)
\rightarrow F^K_{D_{2,\infty}} (\cdot, \cdot) $ uniformly on
compact sets of $ D_{2,\infty} \times \mathbf{C}^2$ (see
\cite{Seshadri&Verma-2006}). Also, note that $ \gamma^j
\rightarrow \gamma $ and $ \dot{\gamma}^j \rightarrow \dot{\gamma}
$ uniformly on $ [0,1]$. Therefore for $j$ large, we obtain
\[
\int_0 ^1 { F^K_{D_2^j} \big( \gamma^j(t), \dot{\gamma}^j(t)
\big)} dt \leq \int_0 ^1 { F^K _{D_{2,\infty}} \big( \gamma(t),
\dot{\gamma}(t) \big)} dt + \epsilon/2 < d_{D_{2,\infty}} (x^0,
z^0 ) + \epsilon.
\]
By definition of $ d_{D_2 ^j} ( x^0, z^j )$ it follows that
\[
d_{D_2 ^j} ( x^0, z^j ) \leq \int_0 ^1 { F^K_{D_2^j} \big(
\gamma^j(t), \dot{\gamma}^j(t) \big)} dt \leq d_{D_{2,\infty}}
(x^0, z^0 ) + \epsilon.
\]
Thus
\begin{eqnarray}
\displaystyle\limsup_{j \rightarrow \infty} d_{D_2 ^j} ( x^0, z^j
) \leq d_{D_{2,\infty}} (x^0, z^0 ). \label{2.1a}
\end{eqnarray}

\noindent Conversely, since $ K \cup \{ x^0 \}$ is a compact
subset of $ D_{2,\infty} $, it follows that $ K \cup \{ x^0 \}$ is
compactly contained  $D_2^j$ for all $j$ large. Fix $ \epsilon >
0$ and let $ V \subset U_2 $ be sufficiently small neighbourhoods
of $ q^0 \in \partial D_2 $ with $ V$ compactly contained in $U_2$
so that
\begin{eqnarray}
F^K_{D_2}( z, v ) \leq F^K_{U_2 \cap D_2} (z,v) \leq ( 1 +
\epsilon) F^K_{D_2}(z, v) \label{2.2a}
\end{eqnarray}

\noindent for $ z \in V \cap D_2 $ and $v$ a tangent vector at $
z$. If $j$ is sufficiently large, $ { (T^j \circ h^j ) }^{-1}
(x^0) $ and $ { (T^j \circ h^j ) }^{-1} (z^j) $ belong to $ V \cap
D_2 $. If $U_2$ is small enough, $U_2 \cap D_2 $ is strictly
convex and it follows from Lempert's work \cite{Lempert-1981} that
there exist $ m_j > 1$ and holomorphic mappings
\[
\phi^j : \Delta(0, m_j) \rightarrow U_2 \cap D_2
\]
such that $ \phi^j(0) = { (T^j \circ h^j ) }^{-1} (x^0),\ \phi^j
(1) = { (T^j \circ h^j ) }^{-1} (z^j) $ and
\begin{eqnarray}
d_{U_2 \cap D_2 } \left( { (T^j \circ h^j ) }^{-1}(x^0), { (T^j
\circ h^j ) }^{-1} (z^j) \right)  & = &
d_{\Delta(0, m_j)} (0,1) \nonumber  \\
& =  & \int_0 ^1 { F^K _{ U_2 \cap D_2 } \big( \phi^j(t),
\dot{\phi}^j (t) \big) } dt . \label{2.2b}
\end{eqnarray}

\noindent By Proposition 3 of \cite{Venturini-1989}, it follows
that
\[
d_{U_2 \cap D_2 } \left( { (T^j \circ h^j ) }^{-1}(x^0), { (T^j
\circ h^j ) }^{-1} (z^j) \right)  \leq (1 + \epsilon) d_{ D_2 }
\left( { (T^j \circ h^j ) }^{-1}(x^0), { (T^j \circ h^j ) }^{-1}
(z^j) \right)
\]
for all $j$ large. Since $ T^j \circ h^j $ are biholomorphisms and
hence Kobayashi isometries,
\begin{eqnarray}
d_{T^j \circ h^j (U_2 \cap D_2)} (x^0, z^j) \leq (1 + \epsilon)
d_{D_2^j}( x^0, z^j). \no
\end{eqnarray}
Now (\ref{2.2b}) shows that
\begin{eqnarray*}
\frac{1}{2} \log \left( \frac{m_j + 1} { m_j - 1} \right) =
d_{\Delta(0, m_j)} (0,1) & =  &
d_{U_2 \cap D_2 } \left( { (T^j \circ h^j ) }^{-1}(x^0), { (T^j \circ h^j ) }^{-1} (z^j)\right) \\
&  = & d_{T^j \circ h^j (U_2 \cap D_2)} (x^0, z^j) \\
& \leq & (1 + \epsilon) d_{D_2^j}( x^0, z^j).
\end{eqnarray*}
However from (\ref{2.1a}) we have that
\[
d_{D_2^j}( x^0, z^j) \leq d_{D_{2,\infty}} (x^0, z^0) + \epsilon <
\infty
\]
and hence $ m_j > 1 + \delta$ for some uniform $ \delta > 0 $ for
all $ j $ large. Thus the holomorphic mappings $ \sigma^j = T^j
\circ h^j \circ \phi^j : \Delta (0, 1 + \delta) \rightarrow T^j
\circ h^j (U_2 \cap D_2) \subset D_2^j $ are well-defined and
satisfy $ \sigma^j(0) = x^0 $ and $ \sigma^j(1) = z^j $.

\medskip

\noindent We claim that $ \{ \sigma^j \} $ admits a subsequence
that converges uniformly on compact sets of $ \Delta (0, 1 +
\delta) $ to a holomorphic mapping $ \sigma: \Delta (0, 1 +
\delta) \rightarrow D_{2,\infty} $. Indeed consider the disc $
\Delta (0, r) $ of radius $ r \in (0,1 + \delta) $. Observe that $
(T^j \circ h^j ) ^{-1} \circ \sigma^j(0) = \phi^j (0)= (T^j \circ
h^j )^{-1}(x^0) \rightarrow q^0 \in \partial D_2 $ as $ j
\rightarrow \infty$. Let $ W $ be a sufficiently small
neighbourhood of $ q^0$. Since $ q^0 \in \partial D_2 $ is a local
peak point, it follows that $(T^j \circ h^j )^{-1} \circ \sigma^j
(\Delta (0,r)) \subset W \cap D_2 $ for all $j$ large. If $W$ is
small enough, there exists $ R > 1$ such that for all $j$ large
\[
h^j ( W \cap D_2) \subset \Big \{ z \in \mathbf{C}^2 : |z_1|^2 + |
z_2 + R|^2 < R^2 \Big \} \subset \Omega_0
\]
where
\[
\Omega_0 = \Big \{ z \in \mathbf{C}^2 : 2 R (\Re z_2) + |z_1 |^2 <
0 \Big \}.
\]
Note that $ \Omega_0$ is invariant under $T^j$ and $ \Omega_0 $ is
biholomorphically equivalent to $ \mathbf{B}^2 $. Hence $
\sigma^j( \Delta (0, r)) \subset T^j \circ h^j ( W \cap D_2)
\subset \Omega_0$ for all $j$ large. If $ \sigma^j(z) = (
\sigma^j_1(z), \sigma^j_2(z)) $ for each $j$, this exactly means
that
\[
2 R \big( \Re (\sigma^j_2 (z)) \big) + | {\sigma}^j_1(z)|^2 < 0
\]
whenever $ z \in \Delta(0,r)$. It follows that $ \{ \sigma^j_2(z)
\} $ and hence $ \{{\sigma}^j_1(z)\}$ forms a normal family on $
\Delta(0, r) $. Since $ r \in (0,1 + \delta) $ was arbitrary, the
usual diagonal subsequence yields a holomorphic mapping $ \sigma:
\Delta (0, 1 + \delta) \rightarrow \mathbf{C}^2 $ or $ \sigma
\equiv \infty$ on $ \Delta (0, 1 + \delta) $. The latter is not
possible since $ \sigma(0) = x^0 $.

\medskip

\noindent It remains to show that $ \sigma: \Delta (0, 1 + \delta)
\rightarrow D_{2,\infty} $. Following \cite{Pinchuk-1980}, note
that $ D_2^j$ are defined by
\begin{eqnarray*}
\rho^j (z) = 2 \ \Re z_2 + | z_1|^2 + A^j (z)
\end{eqnarray*}
where
\[
|A^j(z)| \leq |z|^2 \big( c\sqrt {\epsilon_j} + \eta( \epsilon_j
|z|^2)\big)
\]
and $ \eta(t) $ is a function of one real variable such that $
\eta(t) = o(1) $ as $ t \rightarrow 0 $. Thus for $ z \in
\Delta(0, r)$ and $ r \in (0,1 + \delta) $,
\begin{equation}
2  \Re \big(\sigma^j_2 (z) \big) + |{\sigma}^j_1(z)|^2 + A^j(
\sigma^j(z)) < 0 \label{2.2c}
\end{equation}
where
\[
|A^j(\sigma^j(z))| < | \sigma^j(z) |^2 \big(c \sqrt {\epsilon_j} +
\eta( \epsilon_j |\sigma^j(z)|^2)\big).
\]
Letting $ j \rightarrow \infty $ in (\ref{2.2c}) yields
\[
2 \Re \big(\sigma_2 (z) \big) + |{\sigma}_1(z)|^2 \leq 0
\]
for $ z \in \Delta (0, r) $ or equivalently that $ \sigma
(\Delta(0, r) ) \subset \overline{D}_{2,\infty} $. Since $ r \in
(0,1 + \delta) $ was arbitrary, it follows that $ \sigma (
\Delta(0, 1 + \delta)) \subset \overline{D}_{2,\infty} $. Since $
\sigma(0) = x^0$, the maximum principle shows that $ \sigma (
\Delta(0, 1 + \delta)) \subset D_{2,\infty} $. Using (\ref{2.2a})
and (\ref{2.2b}), we get
\begin{eqnarray*}
\int_0 ^1 { F^K_{D_2^j} \big( \sigma^j(t),  \dot{\sigma}^j(t) \big
)} dt & \leq &  \int_0 ^1 { F^K _{D_2} \big( \phi^j(t),
\dot{\phi}^j(t) \big)}
dt  \\
& \leq & \int_0 ^1 { F^K _{U_2 \cap D_2} \big( \phi^j(t),
\dot{\phi}^j(t) \big)} dt \\
& = & d_{T^j \circ h^j (U_2 \cap D_2)} (x^0, z^j) \\
& \leq & (1 + \epsilon) d_{D_2^j}( x^0, z^j)
\end{eqnarray*}
Since $ \sigma^j \rightarrow \sigma $ and $ \dot{\sigma}^j
\rightarrow \dot{\sigma} $ uniformly on $ [0,1]$, again exploiting
the uniform convergence of $ F^K_{D_2^j}( \cdot, \cdot)
\rightarrow F^K_{D_{2,\infty}} (\cdot, \cdot)$ on compact sets of
$ D_{2,\infty} \times \mathbf{C}^2$, we see that
\begin{eqnarray*}
\int_0 ^1 {F^K_{D_{2,\infty}} \big(\sigma(t), \dot{\sigma}(t)\big)
dt} \leq \int_0 ^1{F^K_{D_2^j} \big(\sigma^j(t), \dot{\sigma}^j(t)
\big) dt } + \epsilon \leq d_{D_2^j}( x^0, z^j) + C \epsilon
\end{eqnarray*}
for all $j$ large. Finally, observe that $ \sigma |_{[0,1]}$ is a
differentiable path in $D_{2,\infty}$ joining $ x^0$ and $ z^0 $.
Hence by definition
\begin{eqnarray}
d_{D_{2,\infty}} (x^0, z^0 ) \leq \int_0 ^1 {F^K_{D_{2,\infty}}
\big(\sigma(t), \dot{\sigma}(t) \big) dt} \leq d_{D_2^j}( x^0,
z^j) + C \epsilon \label{2.2d}
\end{eqnarray}
Combining (\ref{2.1a}) and (\ref{2.2d}) shows that
\[
\displaystyle\lim_{j \rightarrow \infty} d_{D_2 ^j} ( x^0, z^j ) =
d_{D_{2,\infty}} (x^0, z^0 )
\]
which contradicts the assumption (\ref{Q3}) and proves the
required result.
\end{proof}

\begin{prop} \label{B2} Fix $ x^0 \in D_{2,\infty}$ and $ R > 0$. Then
\[
B_{D_2^j} (x^0, R) \rightarrow B_{D_{2,\infty}}(x^0, R)
\]
in the Hausdorff sense. Moreover, for any $ \epsilon > 0 $ and for
all $j$ large

\begin{enumerate}

\item [(i)] $ B_{D_{2,\infty}}(x^0, R) \subset B_{D_2^j} (x^0, R +
\epsilon),$

\medskip

\item [(ii)] $ B_{D_2^j} (x^0, R - \epsilon)  \subset
B_{D_{2,\infty}}(x^0, R).$

\end{enumerate}

\end{prop}

\begin{proof} Let $ K \subset B_{D_{2,\infty}}(x^0, R)$ be compact. Then $ K
$ is a relatively compact subset of $ D_2^j$ for all $j$ large and
there exists a positive constant $ c=c(K) \in (0, R) $ such that $
d_{D_{2,\infty}} (x^0,z) < c$ for all $ z \in K $. Pick $
\tilde{c} \in (c, R) $. It follows from Proposition \ref{D1} that
\[
d_{D_2^j} (x^0, z) \leq d_{D_{2,\infty}} (x^0,z) + \tilde{c} - c
\]
for all $z$ in $K$ and for all $j$ large. Therefore
\[
d_{D_2^j }(x^0, z) \leq \tilde{c} < R
\]
for all $z \in K$ and for all $j$ large. This is just the
assertion that $ K $ is compactly contained in $ B_{D_2^j} (x^0,
R) $ for all $j$ large. Conversely, let $ K \subset \mathbf{C}^2$
be a compact set such that $K $ is compactly contained in $
B_{D_2^j} (x^0, R)$ for all $ j$ large. Then $ K $ is relatively
compact subset of $D_{2,\infty} $. Additionally, there exists a
positive constant $ c= c(K) \in (0, R) $ such that $ d_{D_2^j}
(x^0, z) \leq c $ for all $ z \in K $ and for all $ j$ large. Pick
$ \tilde{c} \in (c, R) $. Again applying Proposition \ref{D1}, we
see that
\[
d_{D_{2,\infty}}(x^0, z) < d_{D_2^j}(x^0,z) + \tilde{c} - c \qquad
\]
for all $ z \in K $ and all $j$ large. Thus for all $ z \in K$, we
obtain
\[
d_{D_{2,\infty}} (x^0, z) < \tilde{c} < R
\]
or equivalently $ K $ is compactly contained in $
B_{D_{2,\infty}}(x^0, R)$. This shows that the sequence of domains
$ \big \{ B_{D_2^j}(x^0, R) \big \}$ converges in the Hausdorff
metric to $ B_{D_{2,\infty}}(x^0, R)$.

\medskip

\noindent To verify (i), first observe that the closure of $
B_{D_{2,\infty}}(x^0, R) $ is compact since $ D_{2,\infty}$ is
Kobayashi complete. Then using Proposition \ref{D1}, we get that
\[
d_{D_2^j} (x^0, z) \leq d_{D_{2,\infty}} (x^0, z) + \epsilon
\]
for all $ z $ in the closure of $ B_{D_{2,\infty}}(x^0, R) $ and
for all $ j $ large. Said differently,
\[
B_{D_{2,\infty}}(x^0, R) \subset B_{D_2^j} (x^0, R + \epsilon)
\]
for all $j$ large.

\medskip

\noindent For (ii) suppose that the desired result is not true.
Then there exists a $ \epsilon_0 > 0$ and a sequence of points $
\{ a^l \} \subset \partial B_{D_{2,\infty}}(x^0, R)$ such that $
a^l \in B_{D_2^l} (x^0, R - \epsilon_0) $. In view of compactness
of $\partial B_{D_{2,\infty}}(x^0, R)$, we may assume that $ a^l
\rightarrow a \in \partial B_{D_{2,\infty}}(x^0, R)$ as $ l
\rightarrow \infty$. It follows from Proposition \ref{D1} that
\[
d_{D_2^l} (a^l, x^0) \rightarrow d_{D_{2,\infty}} (a, x^0)
\]
as $ l \rightarrow \infty$. Consequently, $ d_{D_{2,\infty}} (a,
x^0) \leq R - \epsilon_0 $. This violates the fact that $
d_{D_{2,\infty}} (a, x^0)= R$ thereby proving (ii).
\end{proof}


\noindent Now, consider the composition
\[
f^j = T^j \circ h^j \circ f \circ (\Delta^j)^{-1} : D_1^j
\rightarrow D_2^j
\]
Then $ f^j( 0,-1) = (0,-1) $ for all $ j$. Note that $ f^j $ is
also an isometry for the Kobayashi distances on $ D_1^j $ and $
D_2^j $.

\medskip

\noindent \textbf{Step III:} Let $ \{ K_{\nu} \} $ be an
increasing sequence of relatively compact subsets of $ D_{1,
\infty} $ that exhausts $ D_{1, \infty} $. Fix a pair $ K_{\nu_0}
$ compactly contained in $ K_{\nu_0 + 1}$ such that $ (0,-1) \in
K_{\nu_0} $ and write $ K_1 = K_{\nu_0}$ and $ K_2 = K_{\nu_0 +
1}$ for brevity. Let $ \omega(K_1)$ be a neighbourhood of $ K_1$
such that $ \omega(K_1) \subset K_2$. Since $ \{ D^j_1 \} $
converges to $ D_{1, \infty}$, it follows that $ K_1 \subset
\omega(K_1) \subset K_2 $ which in turn is relatively compact in $
D^j_1 $ for all $j$ large. We show that the sequence $ \{ f^j \}$
is equicontinuous at each point of $ \omega(K_1)$.

\medskip

\noindent Since each $ f^j $ is an Kobayashi isometry, we have
that
\[
d_{D^j_2 } \big( f^j(x), f^j(y) \big) = d_{D^j_1}(x,y)
\]
for all $ j $ and $ x, y $ in $ K_2 $. In particular,
\[
d_{D^j_2 } \big( f^j(x), (0,-1) \big) = d_{D^j_1} \big(x,(0,-1)
\big)
\]
for all $ x $ in $ K_2 $. Now, since
\[
d_{D_1^j} \big( \cdot, (0,-1) \big) \rightarrow d_{D_{1,\infty}}
\big(\cdot, (0,-1) \big)
\]
uniformly on $ K_2 $ which is compactly contained in
$D_{1,\infty}$. Therefore, given $ \epsilon > 0 $
\begin{equation*}
d_{D^j_2 } \big( f^j(x), (0,-1) \big) = d_{D^j_1} \big(x,(0,-1)
\big) \leq d_{D_{1,\infty}} \big(x, (0,-1) \big) + \epsilon < R_0
\end{equation*}
for some $ R_0 > 0 $, for all $ x \in K_2 $ and $ j$ large. This
is just the assertion that
\[
\{ f^j(K_2) \} \subset B_{D^j_2} \big( (0,-1), R_0 \big).
\]
Further, from Proposition \ref{B2} we have
\[
B_{D^j_2} \big( (0,-1), R_0 \big) \subset B_{D_{2, \infty}} \big(
(0,-1), R_0 + \epsilon)
\]
for all $j$ large. It follows that $ \{ f^j (K_2) \} $ is
uniformly bounded.

\medskip

\noindent Now, choose $ R \gg 4 R_0 $ and let $ W $ be a
sufficiently small neighbourhood of $ q^0 \in \partial D_2 $.
Observe that
\[
\big( T^j \circ h^j \big)^{-1}  \Big[ B_{D_{2, \infty}} \big(
(0,-1), R + \epsilon ) \Big] \subset W \cap D_2
\]
for all $j$ large. If $W$ is small enough, there exists $ R' > 1$
such that for all $j$ large
\[
h^j ( W \cap D_2 ) \subset \Big \{ z = (z_1,z_2) \in \mathbf{C}^2
: |z_1|^2 + | z_2 + R'|^2 < R'^2 \Big \} \subset \Omega_0
\]
where
\[
\Omega_0 = \Big \{ z \in \mathbf{C}^2 : 2 R' (\Re z_2) + |z_1 |^2
< 0 \Big \}.
\]
Note that $ \Omega_0$ is invariant under $T^j$ and $ \Omega_0 $ is
biholomorphically equivalent to $ \mathbf{B}^2 $. Hence
\[
B_{D_{2, \infty}} \big( (0,-1), R + \epsilon \big) \subset T^j
\circ h^j (W \cap D_2) \subset \Omega_0
\]
and consequently
\begin{multline} \label{P1}
d_{\Omega_0} \big( f^j(x), f^j(y) \big) \leq d_{B_{D_{2, \infty}}
\big( (0,-1), R + \epsilon \big)} \big( f^j(x), f^j(y) \big) \leq
d_{B_{D^j_2} \big( (0,-1), R \big)} \big( f^j(x), f^j(y) \big)
\end{multline}
for all $x, y $ in $ K_2 $ and for all $j$ large. Now, apply Lemma
\ref{R} to the domain $ D_2^j$ and argue as in the proof of
Proposition \ref{K3}. Let the Kobayashi metric ball $ B_{D_2^j}
\big( (0,-1), R \big)$ play the role of the subdomain $ D'$. Then
\begin{equation*}
d_{B_{D_2^j} \big( (0,-1), R \big)} \big( f^j(x), f^j(y) \big)
\leq \frac{d_{D_2^j} \big( f^j(x), f^j(y) \big)} { \tanh \left( R
- d_{D_2^j} \big( f^j(x), f^j(y) \big) \right)}.
\end{equation*}
Since $ \{ f^j(K_2) \} \subset B_{D^j_2} \big( (0,-1), R_0 \big) $
and $x \rightarrow \tanh x $ is increasing on $ [0,\infty)$, it
follows that
\begin{equation} \label{V2}
d_{B_{D_2^j} \big( (0,-1), R \big)} \big( f^j(x), f^j(y) \big)
\leq \frac{d_{D_2^j} \big( f^j(x), f^j(y) \big)} { \tanh ( R - 2
R_0 )}.
\end{equation}
Moreover, for each $ x \in \omega(K_1)$ fixed, there exists a $r
> 0$ such that $ B( x,r) $ is compactly contained in $ \omega(K_1)$.
The distance decreasing property of the Kobayashi metric together
with its explicit form on $ B(x, r) $ gives
\begin{eqnarray} \label{V3}
d_{D^j_2 } \big( f^j(x), f^j(y) \big) = d_{D^j_1}(x,y) \leq d_{B(
x,r)} (x,y) \leq  | x - y | /c
\end{eqnarray}
for all $j$ large, $ y \in B( x, r) $ and a uniform constant $ c >
0 $.

\medskip

\noindent Combining (\ref{P1}), (\ref{V2}) and (\ref{V3}), we see
that
\[
d_{\Omega_0} \big( f^j(x), f^j(y) \big) \leq \frac{d_{D_2^j} \big(
f^j(x), f^j(y) \big)} { \tanh ( R - 2 R_0 )} \leq \frac{ | x - y
|} {c \tanh{R_0} } .
\]
Now, using the fact that $ \Omega_0 \backsimeq \mathbf{B}^2 $ and
the explicit form of the metric on $ \mathbf{B}^2 $ gives
\[
\big | f^j(x) - f^j(y) \big| \lesssim | x - y |
\]
for $ y \in B(x, r) $. This shows that $ \{ f^j \} $ is
equicontinuous at each point of $ \omega(K_1) $. The diagonal
subsequence still denoted by the same symbols then converges
uniformly on compact subsets of $ D_{1, \infty} $ to a limit
mapping $ \tilde{f} : D_{1, \infty} \rightarrow \overline{D}_{2,
\infty} $ which is continuous.

\medskip

\noindent \textbf{Step IV:} We now show that $ d_{D_{1, \infty}}
(x,y) = d_{D_{2, \infty}} \big( \tilde{f}(x), \tilde{f}(y) \big)$
for all $ x, y \in \Omega_1 $ where $ \Omega_1 = \big \{ z \in
D_{1, \infty} : \tilde{f}(z) \in D_{2, \infty} \big \}$. Note that
$ (0, -1) \in \Omega_1$ and hence $ \Omega_1$ is non-empty. It is
already known that
\begin{equation*}
d_{D^j_1} (x,y) = d_{D^j_2} \big( f^j(x),f^j(y) \big)
\end{equation*}
for all $j$. First note that $ d_{D^j_1} (x,y) \rightarrow
d_{D_{1, \infty}} (x,y) $ as $ j \rightarrow \infty $ as can be
seen from the arguments presented in Proposition \ref{K3} (cf.
Lemma 5.7 of \cite{Mahajan&Verma}). It remains to show that the
right side above converges to $ d_{D_2^j} \big( \tilde{f}(x),
\tilde{f}(y) \big)$ as $ j \rightarrow \infty $. For this observe
that
\[
\big | d_{D^j_2} \big(f^j(x),f^j(y) \big) - d_{D^j_2} \big (
\tilde{f}(x), \tilde{f}(y) \big) \big| \leq   d_{D^j_2} \big(
f^j(x), \tilde{f}(x) \big) + d_{D^j_2} \big( \tilde{f}(y),f^j(y)
\big)
\]
by the triangle inequality. Since $ f^j(x) \rightarrow
\tilde{f}(x) $ and the domains $ D^j_2 $ converge to $ D_{2,
\infty} $, it follows that there is a small ball $ B \big(
\tilde{f}(x), r \big) $ around $ \tilde{f}(x) $ which contains $
f^j(x) $ for all large $ j $ and which is contained in $ D^j_2 $
for all large $j$, where $ r > 0 $ is independent of $j$. Thus
\[
d_{D^j_2} \big( f^j(x), \tilde{f}(x) \big) \lesssim \big| f^j(x) -
\tilde{f}(x) \big |.
\]
The same argument works for showing that $ d_{D^j_2} \big(
\tilde{f}(y),f^j(y) \big) $ is small. So to verify the claim, it
is enough to prove that $ d_{D^j_2} \big ( \tilde{f}(x),
\tilde{f}(y) \big) $ converges to $ d_{D_{2, \infty}} \big(
\tilde{f}(x), \tilde{f}(y) \big) $. But this is immediate from
Proposition \ref{D1}.

\medskip

\noindent \textbf{Step V:} The limit map $ \tilde{f} $ is a
surjection onto $ D_{2,\infty} $. Firstly, we need to show that $
\tilde{f}(D_{1, \infty}) \subset D_{2, \infty}$. Indeed, $
\Omega_1 = D_{1, \infty}$. If $ s^0 \in
\partial \Omega_1 \cap D_{1, \infty} $, choose a sequence $ s^j \in \Omega_1$
that converges to $ s^0$. It follows from Step IV that
\begin{equation*}
d_{D_{1, \infty}} \big( s^j, (0, -1) \big) = d_{D_{2, \infty}}
\big( \tilde{f}(s^j), (0,-1) \big)
\end{equation*}
for all $j$. Since $ s^0 \in \partial \Omega_1$, the sequence $ \{
\tilde{f}(s^j) \} $ converges to a point on $ \partial D_{2,
\infty}$ and as $D_{2, \infty} $ is complete in the Kobayashi
distance, the right hand side above becomes unbounded. However,
the left hand side remains bounded again because of completeness
of $ D_{1, \infty} $. This contradiction shows that $ \Omega_1 =
D_{1, \infty} $ which exactly means that $ \tilde{f}(D_{1,
\infty}) \subset D_{2, \infty}.$ The above observation coupled
with Step IV forces that
\begin{equation*}
d_{D_{1, \infty}} (x,y) = d_{D_{2, \infty}} \big( \tilde{f}(x),
\tilde{f}(y) \big)
\end{equation*}
for all $ x, y \in D_{1, \infty} $. To establish the surjectivity
of $ \tilde{f}$, consider any point $ u^0 \in \partial \big(
\tilde{f}( D_{1, \infty} ) \big) \cap D_2 $ and choose a sequence
$ u^j \in \tilde{f}(D_{1, \infty})$ that converges to $ u^0 $. Let
$ \{ t^j \} $ be sequence of points in $ D_{1, \infty} $ be such
that $ \tilde{f}(t^j)= u^j$. Then for all $j$ and for all $ x \in
D_{1, \infty} $,
\begin{equation} \label{V4}
d_{D_{1, \infty}}(x, t^j) = d_{D_{2, \infty}} \big( \tilde{f}(x),
\tilde{f}(t^j) \big)
\end{equation}
There are two cases to be considered. After passing to a
subsequence, if needed,

\begin{enumerate}

\item[(i)] $  t^j \rightarrow t^0 \in \partial D_{1, \infty} $,

\item [(ii)] $ t^j \rightarrow  t^1 \in D_{1, \infty} $ as $ j
\rightarrow \infty $.

\end{enumerate}

\medskip

\noindent In case (i), observe that the right hand side remains
bounded because of the completeness of $ D_{2, \infty}$. Moreover,
since $ D_{1, \infty}$ is complete in the Kobayashi metric, the
left hand side in (\ref{V4}) becomes unbounded. This contradiction
shows that $ \tilde{f}(D_{1, \infty}) = D_{2, \infty}$.

\medskip

\noindent For (ii), firstly, the continuity of the mapping $
\tilde{f} $ implies that the sequence $ \{ \tilde{f} (t^j) \} $
converges to the point $ \tilde{f}(t^1)$. Therefore, we must have
$ \tilde{f}(t^1) = u^0$. Consider the mappings $ (f^j)^{-1} :
D^j_2 \rightarrow D^j_1$. Now, an argument similar to the one
employed in Step II yields that the sequence $ \{ (f^j)^{-1} \}$
admits a subsequence that converges uniformly on compact sets of $
D_{2, \infty} $ to a continuous mapping $ \tilde{g}: D_{2, \infty}
\rightarrow \overline{D}_{1, \infty}$. Then $ \tilde{g} \circ
\tilde{f} \equiv id_{D_{1, \infty}}$. Therefore,
\[
t^1 = \tilde{g} \circ \tilde{f}(t^1) = \tilde{g}(u^0) =
\displaystyle \lim_{j \rightarrow \infty} (f^j)^{-1} (u^0)
\]
and consequently the sequence $ \{ (f^j)^{-1} (u^0) \} $ is
compactly contained in $ D_{1, \infty}$. Now, repeating the
earlier argument for $ \{(f^j)^{-1}\}$, it follows that $
\tilde{g} : D_{2, \infty} \rightarrow D_{1, \infty} $ and $
\tilde{f} \circ \tilde{g} \equiv id_{D_{2, \infty}} $. In
particular, $ \tilde{f}$ is surjective.

\medskip

\noindent This shows that $ \tilde{f} $ is a continuous isometry
between $ D_{1, \infty} $ and $ D_{2, \infty} $ in the Kobayashi
metric. The goal now is to show that this continuous isometry is
indeed a biholomorphic mapping. To do this, we use ideas from
\cite{Seshadri&Verma-2009}.

\medskip

\noindent \textbf{Step VI:} The Kobayashi distance of $ D_{1,
\infty} $ and the Euclidean distance are Lipschitz equivalent on
any compact convex subset of $ D_{1, \infty} $. This follows from
Lemma 3.3 of \cite{Seshadri&Verma-2009}. It should be mentioned
that although Lemma 3.3 is a statement about strongly convex
domains, the same proof gives the required result in our setting.

\medskip

\noindent \textbf{Step VII:} $ \tilde{f} $ is differentiable
almost everywhere. Since the restriction of $ \tilde{f} $ to any
relatively compact convex subdomain gives a Lipschitz map with
respect to the Euclidean distance, we may apply the classical
theorem of Rademacher-Stepanov to get the required result.

\medskip

\noindent \textbf{Step VIII:} By \cite{Ma-1995}, the infinitesimal
Kobayashi metric $ F^K_{D_{1, \infty}} $ is $ C^1$-smooth on $ {
D_{1, \infty}} \times \mathbf{C}^2 \setminus \{0\} $. Further, $
F^K_{D_{1, \infty}} $ is the quadratic form associated to a
Riemannian metric, $ \tilde{f} $ is $ C^1 $ on $ D_{1, \infty} $
and finally $ \tilde{f} $ is holomorphic or anti-holomorphic.
These statements can be deduced from the arguments in
\cite{Seshadri&Verma-2009} without any additional difficulties. It
follows that
\begin{multline}
\mathbf{B}^2 \simeq D_{1, \infty} = \big \{ (z_1, z_2) \in
\mathbf{C}^2 : 2 \Re z_2 + |z_1|^{2m} < 0 \big \} \\ \simeq
\tilde{D} = \big \{ (z_1,z_2) \in \mathbf{C}^2 : |z_1|^{2m} +
|z_2|^2 < 1 \big \}.
\end{multline}
Let $ F : \mathbf{B}^2 \rightarrow \tilde{D} $ be a biholomorphism
which in addition may be assumed to preserve the origin. Since $
\mathbf{B}^2$ and $ \tilde{D}$ are both circular domains, it
follows that $G$ is linear. This forces that $ 2m =2$.

\medskip

\noindent But this exactly means that there exists a local
coordinate system in a neighbourhood of the origin can be written
as
\[
\big \{( z_1, z_2) \in \mathbf{C}^2 : 2 \Re z_2 + |z_1|^2 + o(
|z_1|^2 + \Im z_2 ) < 0 \big \}
\]
This contradicts the fact that $ p^0 = (0,0)$ is a weakly
pseudoconvex point and proves the theorem. \qed

\begin{rem} Theorem \ref{2} is to be interpreted as a version of
Bell's result (cf. \cite{Bell-1981}) on biholomorphic
inequivalence of a strongly pseudoconvex domain and a smoothly
bounded pseudoconvex domain in $ \mathbf{C}^n $. Here, the end
conclusion of non-existence of a global biholomorphism is replaced
by a global isometry. The question of recovering the theorem for
arbitrary weakly pseudoconvex finite type domains for isometries
seems interesting.
\end{rem}

\begin{thm} \label{W}
Let $ f : D_1 \rightarrow D_2 $ be a continuous Kobayashi isometry
between two bounded domains in $ \mathbf{C}^2$. Let $ p^0 $ and $
q^0 $ be points on $ \partial D_1 $ and $ \partial D_2 $
respectively. Assume that the boundaries $ \partial D_1 $ and $
\partial D_2 $ are both $ {C}^{\infty}$-smooth weakly pseudoconvex and of finite type
near $ p^0 $ and $ q^0 $ respectively. Suppose that $ q^0 $
belongs to the cluster set of $p^0 $ under $f$. Then $f$ extends
as a continuous mapping to a neighbourhood of $ p^0 $ in $
\overline{D}_1 $.
\end{thm}

\begin{proof} If $f$ does not extend continuously to any
neighbourhood $ U_2 $ of $ p^0 $ in $ \overline{D}_1 $, there
exists a sequence of points $ \{ s^j \} \subset D_1 $ converging
to $ p^0 \in \partial D_1 $ such that the sequence $ \{ f(s^j ) \}
$ does not converge to the point $ q^0 \in \partial D_2 $. By
hypothesis, there exists a sequence $ \{ p^j \} \subset D_1 $
converging to $ p^0 \in \partial D_1 $ such that the  $
\displaystyle \lim_{j \rightarrow \infty} f(p^j) = q^0 \in
\partial D_2 $. Then for polygonal paths $ \gamma ^j $ in $ D_1 $
joining $ p^j $ and $ s^j $ defined as in Theorem \ref{1} and
points $ p^{j0}, s^{j0}, t^j, u^j $ and $ u^0 $ chosen
analogously, it follows from \cite{Forstneric&Rosay-1987} that
\begin{multline} \label{W1}
d_{D_1}(p^j,t^j)  \leq  - (1/2) \log {d(p^j,\partial D_1)} +
 (1/2) \log \big( d(p^j,\partial D_1) + |p^j - t^j| \big) \\
 + (1/2) \log \big( d(t^j,\partial D_1) + |p^j - t^j| \big) - (1/2) \log {d(t^j,\partial D_1)} +
 C_1
\end{multline}
Applying Proposition \ref{Q} yields
\begin{equation} \label{W2}
d_{D_2} \big( f(p^j), f(t^j) \big) \geq - (1/2) \log d \big(
f(p^j), \partial D_2 \big) - (1/2) \log d \big( f(t^j), \partial
D_2 ) - C_2
\end{equation}
for all $j$ large and uniform positive constant $ C_2 $.  Next, we
claim that
\[ d \big( f(p^j), \partial D_2 ) \leq C_4 d( p^j, \partial D_2 ) \ \mbox{and} \ d \big( f(t^j),
\partial D_2 ) \leq C_4 d( t^j, \partial D_2 )
\]
for some uniform positive constant $ C_4 $. Assume this for now.
Now, using the fact $ d_{D_1} ( p^j, t^j) = d_{D_2} \big( f(p^j),
f(t^j) \big) $ and comparing the inequalities (\ref{W1}) and
(\ref{W2}), it follows from the above claim that for all $j$ large
\[
- ( C_1 + C_2 + \log C_4 ) \leq (1/2) \log \big( d(p^j,\partial
D_1) + |p^j - t^j| \big) + (1/2) \log \big( d(t^j,\partial D_1) +
|p^j - t^j| \big)
\]
which is impossible.

\medskip

\noindent To prove the claim, fix $ a \in D_1 $. By Proposition
\ref{S}, we have that
\begin{equation} \label{W4}
d_{D_1} (a, p^j) \geq -(1/2) \log {d(p^j, \partial D_1 ) } - C_5
\end{equation}
and
\begin{equation} \label{W5}
d_{D_2} \big( f(p^j), f(a) \big) \leq -(1/2) \log d \big(
d(f(p^j), \partial D_2 \big) + C_6.
\end{equation}
for all $j$ large and uniform positive constants $ C_5 $ and $ C_6
$. Fix $ a \in D_1 $, using $ d_{D_1} (a, p^j) = d_{D_2} \big(
f(p^j), f(a) \big) $, and comparing the inequalities (\ref{W4})
and (\ref{W5}), we get the required estimates. Hence the claim.
\end{proof}

\begin{thm} \label{L}
Let $ f : D_1 \rightarrow D_2 $ be a continuous Kobayashi isometry
between two bounded domains in $ \mathbf{C}^n$. Let $ p^0 $ and $
q^0 $ be points on $ \partial D_1 $ and $ \partial D_2 $
respectively. Assume that the boundaries $ \partial D_1 $ and $
\partial D_2 $ are both $ {C}^{2}$-smooth strongly pseudoconvex
near $ p^0 $ and $ q^0 $ respectively. Suppose that $ q^0 $
belongs to the cluster set of $p^0 $ under $f$. Then $f$ extends
as a continuous mapping to a neighbourhood of $ p^0 $ in $
\overline{D}_1 $.
\end{thm}

\noindent The proof of the above theorem is along the same lines
as that of Theorem \ref{1} and Theorem \ref{W} and is hence
omitted.

\medskip

\noindent It turns out that versions of the above mentioned
results hold for the inner Carath\'{e}odory distance. More
concretely, the following global statements can be proved:

\begin{thm} \label{U}
Let $ f : D_1 \rightarrow D_2 $ be a continuous isometry between
two bounded domains in $ \mathbf{C}^n $ with respect to the inner
Carath\'{e}odory distances on these domains.

\begin{enumerate}

\item [(i)] Assume that $ D_1 $ and $ D_2 $ are both $
{C}^{3}$-smooth strongly pseudoconvex domains in $ \mathbf{C}^n $,
then $f$ extends continuously up to the boundary.

\item[(ii)] Assume that $ D_1 \subset \mathbf{C}^2 $ is a $ C^3
$-smooth strongly pseudoconvex domain and $ D_2 \subset
\mathbf{C}^2 $ is a $ C^{\infty}$-smooth weakly pseudoconvex
finite type domain, then $f$ extends continuously up to the
boundary.

\item [(iii)] Assume that $ D_1 $ and $ D_2 $ are both $
C^{\infty}$-smooth weakly pseudoconvex finite type domains in $
\mathbf{C}^2 $, then $f$ extends continuously up to the boundary.

\end{enumerate}

\end{thm}

\noindent To establish the above theorem, the following result due
to Balogh-Bonk (\cite{Balogh&Bonk-2000}) will be needed. This in
turn relies on estimates for the infinitesimal Carath\'{e}odory
metric given by D. Ma (\cite{Ma-1991}).

\begin{thm} \label{P}
Let $ D \subset \mathbf{C}^n $ be a bounded strongly pseudoconvex
domain with $ C^3$-smooth boundary. Then there exists a constant $
C > 0 $ such that for all $a, b \in D $
\[
g(a,b) - C \leq c_D^i(a,b) \leq g(a,b) + C
\]
where
\[
g(a,b) = 2 \log  \left [ \frac{ d_H\big( \pi(a), \pi(b) \big) +
\max \big \{ \big(d(a, \partial D) \big)^{1/2}, \big( d(b,
\partial D) \big)^{1/2} \big \} }{ \big( d(a,
\partial D) \big)^{1/4} \big( d(b, \partial D) \big)^{1/4}  }
\right ],
\]
$ d_H $ is the Carnot-Carath\'{e}odory metric on $ \partial D $
and $ \pi(z) \in \partial D $ such that $ | \pi(z) - z | = d(z,
\partial D ) $.
\end{thm}

\noindent \textit{Proof of Theorem \ref{U}:} Part (i) has been
done in Lemma 2.2 of \cite{Seshadri&Verma-2006}. Parts (ii) and
(iii) can be verified by making the relevant changes in the proof
of Theorem \ref{W} -- using the inequality $ c_D^i \leq d_D $ to
get the upper bounds on the inner Carath\'{e}odory distance and
the following consequences of the results of Balogh-Bonk
(\cite{Balogh&Bonk-2000}) and Herbort (\cite{Herbort-2005}):

\begin{itemize}

\item When $ a$ belongs to some fixed compact set $ L $ in $ D$,
then
\[
g(a,b) \approx -(1/2) \log d(b, \partial D) \pm C(L).
\]
Thus
\[
c_D^i(a,b) \approx -(1/2) \log d(b, \partial D) \pm C(L).
\]

\item In the case when $a, b \in D $ are close to two distinct
points on $ \partial D $, then
\[
g(a,b) \approx -(1/2) \log d(a, \partial D) -(1/2) \log d(b,
\partial D) \pm C.
\]
Consequently,
\[
c_D^i(a,b) \approx -(1/2) \log d(a, \partial D) -(1/2) \log d(b,
\partial D) \pm C.
\]

\item When $ a, b $ are close to the same boundary point,
\begin{multline*}
\ \ \  \qquad g(a,b)   \leq   - (1/2) \log {d(a,\partial D)} +
 (1/2) \log \big( d(a,\partial D) + |a - b| \big)  \\
  +  (1/2) \log \big( d(b,\partial D) + |a - b| \big) - (1/2) \log {d(b,\partial D)} + C
\end{multline*}
Therefore,
\begin{multline*}
\ \ \  \qquad c_D^i(a,b)   \leq   - (1/2) \log {d(a,\partial D)} +
 (1/2) \log \big( d(a,\partial D) + |a - b| \big)  \\
  +  (1/2) \log \big( d(b,\partial D) + |a - b| \big) - (1/2) \log {d(b,\partial D)} + C
\end{multline*}

\item When $a, b $ are sufficiently close to two distinct points
on $ \partial D $ for a weakly pseudoconvex finite domain in $
\mathbf{C}^2 $, then $ \tilde{d}(a,b) \gtrsim 1 $ and $
\tilde{d}(b,a) \gtrsim 1 $ so that
\[
c_D^i(a,b) \gtrsim -(1/2) \log d(a, \partial D) -(1/2) \log d(b,
\partial D) \pm C.
\]

\end{itemize}

\noindent These bounds on $ c_D^i(a,b) $ are exactly the ones that
are needed to reprove Theorem \ref{W} for the inner
Carath\'{e}odory distance. \qed


\end{document}